\documentclass[11pt]{article}

\usepackage{amsmath,amsthm,amssymb,graphicx}
\usepackage{hyperref}
\usepackage{cleveref}
\usepackage[none]{hyphenat}
\usepackage{tocloft}
\usepackage{cancel}
\usepackage{multicol}

\usepackage{ifpdf}

\ifpdf
\else
\fi

\DeclareMathOperator{\lastBorn}{\ell}
\DeclareMathOperator{\choosePath}{c}

\usepackage[all,cmtip]{xy}

\newcommand{\sm}{\setminus}

\binoppenalty=\maxdimen
\relpenalty=\maxdimen
\usepackage{indentfirst}
\usepackage{enumitem}
\setlist[itemize]{noitemsep, topsep=0pt}

\newtheorem{theorem}{Theorem}[section]
\newtheorem{lemma}[theorem]{Lemma}

\newtheorem{question}{Question}

\newcommand{\RR}{\mathbb{R}}

\title{Burling graphs revisited,\\ part II: Structure} \author{Pegah
  Pournajafi\thanks{Univ Lyon, EnsL, UCBL, CNRS, LIP, F-69342, LYON
    Cedex 07, France. Partially supported by the LABEX MILYON
    (ANR-10-LABX-0070) of Universit\'e de Lyon, within the program
    ‘‘Investissements d'Avenir’’ (ANR-11-IDEX-0007) operated by the
    French National Research Agency (ANR) and by Agence Nationale de
    la Recherche (France) under research grant ANR DIGRAPHS
    ANR-19-CE48-0013-01.}~~and Nicolas Trotignon\footnotemark[1]}

\begin{document}
	
\maketitle
	
\begin{abstract}

	The Burling sequence is a sequence of triangle-free graphs of
	increasing chromatic number. Any induced subgraph of a graph in this sequence is called a Burling graph. These graphs have attracted some attention because on one hand they have geometric representations, and on the other hand they provide counter-examples to several conjectures about bounding the chromatic
	number in classes of graphs.
	
	Using an equivalent definition from the first part of this work (called derived graphs), we study several structural properties of Burling graphs. In particular, we give decomposition theorems for the class using in-star cutsets, study holes and their interactions in Burling graphs, and analyze the effect of subdividing some arcs of a Burling graph. Using mentioned results, we introduce new techniques for providing new triangle-free graph that are not Burling graphs.

	Among other applications, we prove that wheels are not Burling graphs. This answers an open problem of the second author and disproves a conjecture of Scott and Seymour. 

\end{abstract}

\section{Introduction}

Graphs in this paper have neither loops nor multiple edges. In this
introduction they are non-oriented, but oriented graphs will sometimes
be considered in the rest of the paper.  A class of graphs is
\emph{hereditary} if it is closed under taking induced subgraphs.  A
\emph{triangle} in a graph is a set of three pairwise adjacent
vertices, and a graph is \emph{triangle-free} if it contains no
triangle.

In 1965, Burling~\cite{Burling65} proved that triangle-free
intersection graphs of axis-parallel boxes in $\mathbb{R}^3$ have
unbounded chromatic number.  The work of Burling uses geometric
terminology. However, one may rephrase it into a more modern graph
theoretic setting by defining in a combinatorial way a sequence of
triangle-free graphs, called the \emph{Burling sequence}, with
increasing chromatic number, and proving that each of them is isomorphic
to the intersection graph of a set of axis-parallel boxes in
$\mathbb{R}^3$. It was later proved that every graph in the Burling
sequence is isomorphic to the intersection graph of various
geometrical objects, see~\cite{Chalopin2014,Pawlik2012Jctb,P2022,PT:2021}.

The Burling sequence also attracted attention lately because it is a
good source of non-$\chi$-bounded classes of graphs. Let us explain
this.  A heredatiry class of graphs is \emph{$\chi$-bounded} if there
exists a function $f$ such that for every graph $G$ in the class,
$\chi(G)\leq f(\omega(G)$, where $\chi(G)$ denotes the chromatic
number of $G$ and $\omega(G)$ the maximum number of pairwise adjacent
vertices in $G$.  In~\cite{Pawlik2012Jctb}, it is shown that there
exist graphs $H$ such that graphs in the Burling sequence (and
therefore some triangle-free graphs with arbitrary large chromatic
number) do not contain any subdivision of $ H $ as an induced
subgraph. This provides counter-examples to a well-studied conjecture
of Scott (Conjecture 8 in~\cite{Scott97}), saying that for every $H$,
the class of graphs that do not contain any subdivision of $ H $ as an
induced subgraphs are $\chi$-bounded. A counter-example to Scott
conjecture is called a \emph{non-weakly pervasive graph}. Despite the
fact that Scott's conjecture is disproved in~\cite{Pawlik2012Jctb},
classifying graphs into weakly pervasive and non-weakly pervasive
remains of interest.

This paper is the second part in a series of three articles, so let us
explain briefly the content of the other parts. In all the series, we
call \emph{Burling graph} any graph that is isomorphic to an induced
subgraph of some graph in the Burling sequence. In the first part of
this work~\cite{PT:2021}, we give several equivalent definitions of
Burling graphs, including geometrical characterizations. In the third
part of this work~\cite{PT-3:2021}, our goal is to use results from
the previous parts to make progress toward understanding weakly
pervasive graphs by giving new examples of graphs that are not weakly
pervasive.

In this second part, we use one of the new equivalent definitions
from the first part~\cite{PT:2021} (namely derived graphs) to study the structure of
Burling graphs. Among other results, we give a decomposition theorem for oriented Burling graphs, characterize the subdivisions of $K_4 $ that are Burling graphs, and give a new characterization of Burling graphs (called $k$-sequential graphs). We also 
answer some open questions in $\chi$-boundedness: an open problem of the second author in~\cite{Trotignon13} and a conjecture of Scott and Seymour in~\cite{ScottSeymour17} (both concerning wheel-free graphs). Moreover, most of the results in this part have applications in the third part~\cite{PT-3:2021} as well. We postpone the explanations for these applications to part III, and in the following brief outline of the paper, we explain how the structural results of this contribution are motivated in their own rights.

In section~\ref{sec:notation}, we describe the notation that we use.

In~section~\ref{section:Derived-graphs}, we recall the definition of
derived graphs (first defined in Part I of this work~\cite{PT:2021}
where they are shown to be equivalent to Burling graphs) which are
oriented graphs obtained from a tree by some precise rules. We also
show how their arcs can be subdivided, providing generic
constructions of Burling graphs.  This study of subdivisions seems
important to us, because the class of Burling graphs is so far the
only source of non-weakly pervasive graphs.

In section~\ref{sec:kBurling}, we define the notion of nobility of a
Burling graph (roughly, it is the maximum number of vertices of the
graph that lie in a branch of a tree from which the graph can be
derived). We then show how Burling graphs of a given nobility can be
constructed from graphs of smaller nobility.  The section is mostly
motivated by the need of the so-called 2-Burling graphs that serve as
basic graphs in the decomposition theorem of the next section. 
We also present $k$-Burling graphs that enables us to provide another equivalent definition of Burling graphs.

In Section~\ref{sec:starCutsets}, we give a decomposition theorem for
Burling graphs using star cutsets.  This theorem was already proved by
Chalopin, Esperet, Li and Ossona de Mendez in~\cite{Chalopin2014}.
Our contribution is to give a stronger result for oriented Burling
graphs that we need in this part (Section~\ref{sec:holes}) and in
part~3.  Interestingly, star cutsets have been proved very useful in
the proof of the strong perfect graph theorem. They were invented by
Chv\'ata~\cite{chvatal:starcutset}l who generalized them to skew
partitions, and who proved that in some sense, star cutsets preserve
being perfect.  It is not the place here to give too much detail about
that, but briefly speaking, many classes of perfect graphs, including
the class of all perfect graphs, have decomposition theorems~\cite{chudnovsky.r.s.t:spgt} (a graph
in the class is basic or has a decomposition), and star cutsets or one
of their variants (skew partitions, double star cutsets, etc.) are
often one of the decompositions.  It is therefore striking to see a
class of triangle-free graphs of high chromatic number that can still
be decomposed into very simple graphs by star cutsets.  Moreover, we
observe here for the first time that the decomposition theorem of
Burling graphs in \cite{Chalopin2014} disproves an open question
from~\cite{CPST:subst} about star cutset preserving in some sense
bounds on the chromatic number.  We believe that this explains why the
decomposition method that was so efficient to prove the strong perfect
graph theorem did not provide many results in the more general field
of $\chi$-boundedness.

In section~\ref{sec:holes}, we describe properties of holes in Burling
graphs, where a hole is a chordless cycle of length at
least~4. Results from this section provide essential tools for the
proofs in Section~\ref{sec:examples} and for the results in part~3 of
this work.  In particular, the orientation of a hole in a Burling
graph is used intensively in the sequel. Also, how a graph is organized
regarding its holes is very often of interest in studying the
structure of many classes of graphs, so we believe that our study might
help future works about the structure of Burling graphs as well. 

In section~\ref{sec:examples}, using the results of our structural studies, we give several examples of graphs that
are not Burling.  Among these examples are wheels (graphs build from a
hole and a vertex with at least three neighbors in
that hole). This results proves that the class of wheel-free graphs is
not $\chi$-bounded, which answers negatively a question of the second
author in~\cite{Trotignon13}, and disproves a conjecture of Scott and
Seymour in~\cite{Scott18}. This result is independently proved by
Davies~\cite{davies2021trianglefree} (with a different
technique) and has been claimed by Scott and
Seymour~\cite{ScottSeymour17} in a personal communication. However, our proof, appearing first in the master's thesis of the first author \cite{report} is the first written proof of this theorem. 

To sum up, besides answering some open problems, we study here many structural aspects of Burling graphs, most
notably, their closure under subdivision, decomposition theorem,
structure and holes. We believe that these studies, along with the examples we provide, might help solving the following
questions.

\begin{question} 
	What is the list of minimal forbidden induced subgraphs that defines Burling graphs?
\end{question}

\begin{question}
	What is the complexity of deciding whether a graph is  Burling or not?
\end{question}

\section{Notation}
\label{sec:notation}

We denote by $ N(v) $ the neighborhood of $ v $, and denote by
$ N[v] $ the closed neighborhood of $ v $, that is $N(v) \cup \{v\} $. If the graph is oriented, $ N^-(v) $ and
$ N^+(v) $ denote the set of in-neighbors and out-neighbors of $ v $
respectively, and $ N^-[v] $ and $ N^+[v] $ are the closed versions of
them.

There is a difficulty regarding notations in this paper. The graphs we
are interested in will be defined from trees. More specifically, a
tree $T$ is considered and a graph $G$ is \emph{derived} from it by
following some rules that are defined in the next section.  We have $V(G) = V(T)$
but $E(G)$ and $E(T)$ are different (disjoint, in fact). Also, even if
we are originally motivated by non-oriented graphs, it turns out that
$G$ has a natural orientation, and considering it is essential in many
of our proofs.
	
So, in many situations we have to deal simultaneously with the tree,
the oriented graph derived from it and the underlying graph of this
oriented graph. A last difficulty is that since we are interested in
hereditary classes, we allow removing vertices from $G$.  But we have
to keep all vertices of $T$ to study $G$ because the so-called
\emph{shadow vertices}, the ones from $T$ that are not in $G$, capture
essential structural properties of $G$.  All this will be clearer in
the next section. For now, it explains why we need to be very careful
about the notation. 
For any classical notion that we do not define, refer to \cite{BondyMurty}.
	
\subsection*{Notation for trees}
	
A \emph{tree} is a graph $T$ such that for every pair of vertices
$u, v\in V(T)$, there exists a unique path from $u$ to $v$.  A
\emph{rooted tree} is a pair $(T, r)$ such that $T$ is a tree and
$r\in V(T)$.  The vertex $r$ is called the \emph{root} of $(T, r)$.
By abuse of notation, we often refer to the rooted tree $(T, r)$ as $T$, in
particular when $r$ is clear from the context.
	
In a rooted tree, each vertex $v$ except the root has a unique
\emph{parent} which is the neighbor of $v$ in the unique path from the
root to $v$. If $u$ is the parent of $v$, then $v$ is a \emph{child}
of $u$.  A \emph{leaf} of a rooted tree is a vertex that has no child.
Note that every tree has at least one leaf.

A \emph{branch} in a rooted tree is a path $ v_1 v_2 \dots v_k $ such
that for each $1\leq i < k$,~$ v_i$ is the parent of $v_{i+1}$.  This
branch \emph{starts} at $v_1$ and \emph{ends} at $v_k$.  A branch
that starts at the root and ends at a leaf is a \emph{principal}
branch.  Note that every rooted tree has at least one principal
branch.
	
If $T$ is a rooted tree, the \emph{descendants} of a vertex $v$ are
all the vertices that are on a branch starting at $v$. The
\emph{ancestors} of $v$ are the vertices on the unique path from $v$
to the root of~$T$.  Notice that a vertex is a descendant and an
ancestor of itself. We call \emph{a strict descendant} (resp.\ \emph{a strict ancestor}) of a vertex any descendant (resp.\ ancestor) other than the vertex itself.

We will no more use words such as
\emph{neighbors}, \emph{adjacent}, \emph{path}, etc for trees. Only
\emph{parent}, \emph{child}, \emph{branch}, \emph{descendant} and
\emph{ancestor} will be used.

\subsection*{Notation for graphs and oriented graphs}

Graphs in this article have no loops and no multiple edges. 
From now on, we use the term \emph{oriented graph} to refer to a graph whose edges
(called \emph{arcs}) are all oriented and no arc is oriented in
both directions, and we use the term \emph{non-oriented graph} or simply \emph{graph} only to refer to a graph with none of its edges oriented.

When $G$ is a graph or an oriented graph, we denote
by $V(G)$ its vertex set.  We denote by $E(G)$ the set of edges of a
graph $G$ and by $A(G)$ the set of arcs of an oriented graph.  When
$u$ and $v$ are vertices, we use the same notation $uv$ to denote an
edge and an arc.  Observe the arc $uv$ is different from the arc $vu$,
while the edge $uv$ is equal to the edge $vu$.
	
When $G$ is an oriented graph, its \emph{underlying graph} is the graph
$H$ such that $V(H) = V(G)$ and for all $u, v\in V(H)$, $uv\in E(H)$
if and only if $uv\in A(G)$ or $vu\in A(G)$.  We then also say that
$G$ is an \emph{orientation of $H$}.  When there is no risk of
confusion, we often use the same letter to denote an oriented graph
and its underlying graph.
	
In the context of oriented graphs, we use the words
\emph{in-neighbor}, \emph{out-neighbor}, \emph{in-degree},
\emph{out-degree}, \emph{sink} and \emph{source} with their classical meaning.
Terms from the non-oriented realm, such as \emph{degree},
\emph{neighbor}, \emph{isolated vertex} or \emph{connected component},
when applied to an oriented graph, implicitly apply to its underlying
graph.

\subsection*{Notation for paths and cycles}

For $k\geq 0$, a \emph{path} of \emph{length} $k$ is a graph $P$ with vertex-set
$\{p_0, \dots, p_{k}\}$ and edge-set $\{p_0p_1, \dots, p_{k-1}p_{k}\}$.
We denote it by $P_k$.  For $k\geq 3$, a \emph{cycle} of length $k$ is a graph $C$
with vertex-set $\{c_1, \dots, c_k\}$ and edge-set
$\{c_1c_2, \dots, c_{k-1}c_k, c_kc_1\}$.  We denote it by $C_k$.
Since we deal mostly with the ``induced subgraph'' containment
relation, when we say that \emph{$X$ is a path (resp.\ a cycle) in some graph $G$}, we mean that $X$ is an induced subgraph of $G$ that is isomorphic to a path~(resp.\ a cycle).

We define directed paths and cycles similarly (the arc-sets are
$\{p_1p_2, \dots, p_{k-1}p_k\}$ and
$\{c_1c_2, \dots, c_{k-1}c_k, c_kc_1\}$ respectively). 

A \emph{path} (resp.\ a \emph{cycle}) in an oriented graph is a path
(resp.\ a cycle) of its underlying graph. Possibly, it is not a
directed path (resp.\ cycle), it can have any orientation.  Notations
$P_k$ and $C_k$ will be used only in the context of non-oriented
graphs.  We use no specific notation for directed paths and cycles.  A
\emph{hole} in a graph $G$ is a cycle of length at least~4 that is an
induced subgraph of $G$.  A \emph{hole} in an oriented graph means a
hole of its underlying graph.

\section{Derived graphs}
\label{section:Derived-graphs}
	
In this section, we recall the definition and main properties of
\emph{derived graphs} that were introduced in \cite{PT:2021}.  A
\emph{Burling tree} is a 4-tuple $(T,r, \lastBorn, \choosePath)$ in
which:
	
\begin{enumerate}[label=(\roman*)]
\item $ T $ is a rooted tree and $ r $ is its root,
\item $\lastBorn$ is a function associating to each vertex $v$ of $T$
  which is not a leaf, one child of $v$ which is called the
  \emph{last-born} of $v$,
\item $\choosePath$ is a function defined on the vertices of $ T $. If
  $ v $ is a non-last-born vertex of $ T $ other than the root, then
  $ \choosePath $ associates to $ v $ the vertex-set of a (possibly
  empty) branch in $T$ starting at the last-born of $p(v)$. If $v$ is
  a last-born or the root of $T$, then we define
  $ \choosePath(v) = \varnothing $. We call $ \choosePath $ the
  \textit{choose function} of $ T $.
\end{enumerate}
	
By abuse of notation, we often use $T$ to denote the 4-tuple.

The oriented graph $G$ \emph{fully derived} from the Burling tree $T$
is the oriented graph whose vertex-set is $V(T)$ and such that
$uv \in A(G) $ if and only if $v$ is a vertex in $\choosePath(u)$.  A
non-oriented graph $G$ is \emph{fully derived} from $T$ if it is the
underlying graph of the oriented graph fully derived from $T$.
	
A graph (resp.\ oriented graph) $G$ is \emph{derived} from a Burling
tree $T$ if it is an induced subgraph of a graph (resp.\ oriented
graph) fully derived from $T$. The graph (resp.\ oriented graph) $G$
is called a \emph{derived graph} if there exists a Burling tree $T$
such that $G$ is derived from $T$.
	
Observe that if the root of $T$ is in $V(G)$, then it is an isolated
vertex of~$G$.  Observe that a last-born vertex of~$T$ that is in $G$
is a sink of $G$. This does not mean that every oriented derived graph
has a sink, because it could be that no last-born of $T$ is in $V(G)$.

In 1965, Burling \cite{Burling65} introduced a sequence $\{G_k\}_{k \geq 1}$ of graphs such that $\chi(G_k) = k$ and where each $ G_k $ is a triangle-free intersection graph of some boxes (i.e.\ axis parallel cuboids) in $ \RR^3$ (see \cite{Burling65} or \cite{PT:2021} for the definition). 
The \emph{class of Burling graphs} is the class of all induced subgraphs of the graphs in the sequence $ \{G_k\}$. 
It was proved in~\cite{PT:2021} that a graph is a Burling graph if and
only if it can be derived from a Burling tree (Theorem 4.9 of \cite{P2022}).  So in what follows, we use
``Burling graph'' or ``graph derived from a tree'' according to what
is most convenient.

Now, let us give some examples of derived graphs. We use the convention in all figures in this paper that the
tree~$T$ is represented with black edges while the arcs of $G$ are
represented in red. The last-born of a vertex of $T$ is its rightmost
child. Moreover, denoting a vertex of~$T$ in white means that this
vertex in not in~$G$.
	
\begin{figure}
  \begin{center}
    \includegraphics[width=7cm]{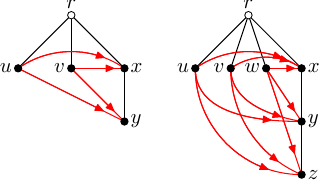}
    \caption{Complete bipartite graphs derived
      from trees.\label{f:square}}
  \end{center}
\end{figure}
	
In the first graph represented in Figure~\ref{f:square}, we have
$\choosePath(u) = \choosePath(v) = \{x, y\} $.  It shows that there exists an orientation of $C_4$ which is an oriented derived graph, so $C_4$ is
a derived graph.  The second graph shows that $K_{3, 3}$ is a derived
graph, and it is easy to generalize this construction to $K_{n, m}$
for all integers $n, m \geq 1$. In both graphs, the vertex $r$ of
$T$ is not a vertex of $G$. Figure~\ref{f:c6} is a presentation of
$ C_6 $ as a derived graph. Notice that in this presentation, the
vertex $v$ is not in~$G$.
	
\begin{figure}
  \begin{center}
    \includegraphics[width=4cm]{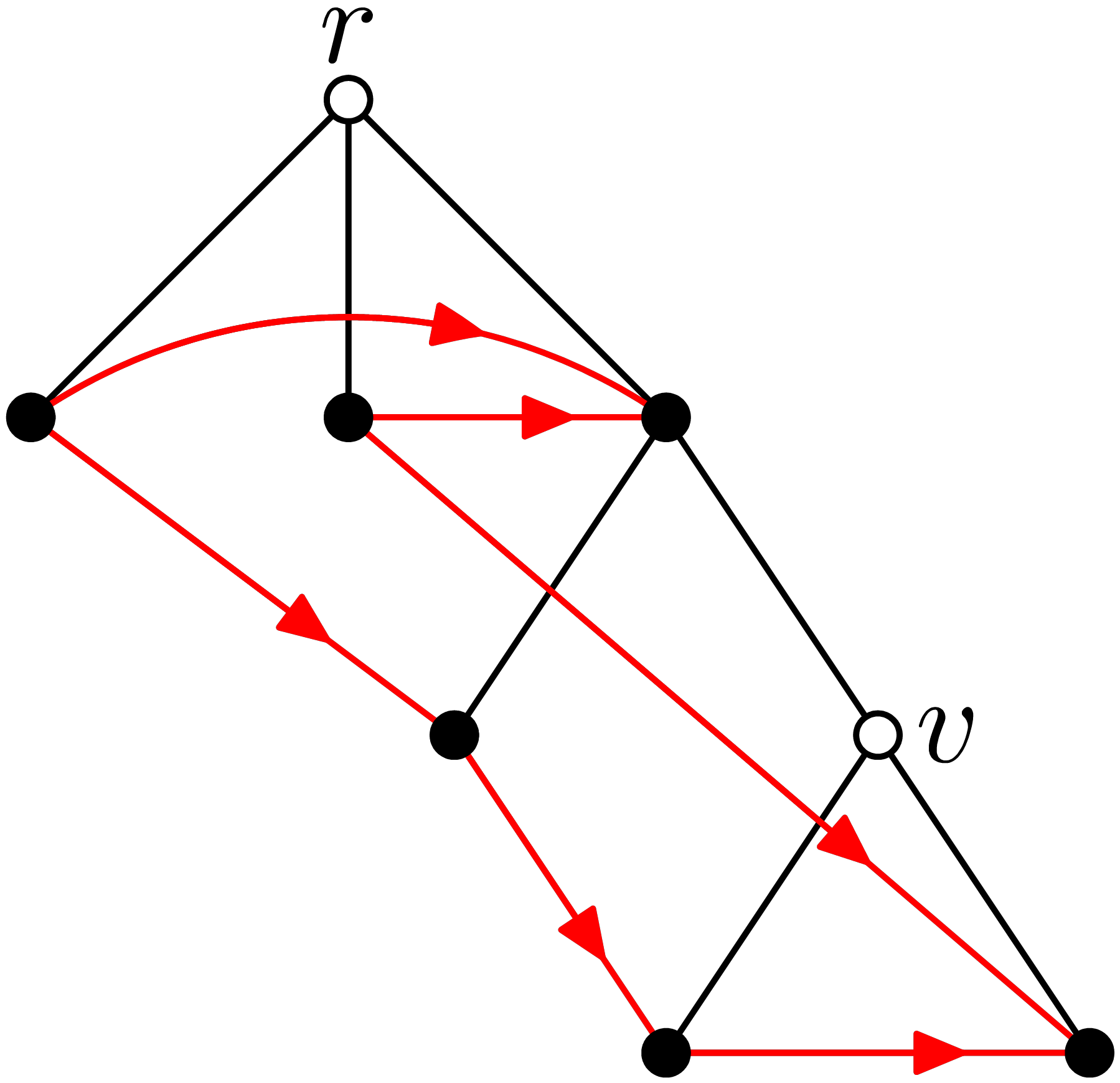}
    \caption{Cycle of length 6  derived from a tree.\label{f:c6}}
  \end{center}
\end{figure}

Notice that if a graph $G$ is derived from $T$, the branches of $T$,
restricted to the vertices of $G$, are stable sets of $G$. In
particular, no edge of $T$ is an edge of $G$.

\begin{lemma}
  \label{lem:simple}
  Suppose that $G$ is an oriented graph derived from a Burling tree
  $T$. If $uv \in A(G)$, then $p(u)$ is an ancestor of $p(v)$. 
\end{lemma}
		
\begin{proof}
  Follows directly from the definition of derived graphs. 
\end{proof}

\begin{lemma}
  \label{lem:DG-no-cycle}
  An oriented Burling graph has no directed cycles.
\end{lemma}

\begin{proof}
  Let $ G $ be a Burling graphs. It is therefore a derived graph (by Theorem 4.9 of \cite{PT:2021}). Assume that $ G $ is derived from a Burling tree $ T $. If a directed cycle $v_1\dots v_kv_1$ exists in $G$, then by Lemma~\ref{lem:simple}, for
  every $1\leq i \leq k$, $p(v_i)$ is an ancestor of $p(v_{i+1})$
  (with addition modulo $k$). This is possible only if
  $p(v_1) = \cdots = p(v_k)$.  So, all $v_i$'s have the same parent
  $p$, so we may assume up to symmetry that $v_1$ and $v_2$ are not
  last-borns of $p$.  This is a contradiction since $v_1v_2\in A(G)$.
\end{proof}

\begin{lemma}
  \label{l:source}
  Every oriented Burling graph contains a source.
\end{lemma}

\begin{proof}
  Follows from Lemma~\ref{lem:DG-no-cycle}.
\end{proof}

\begin{lemma}
  \label{lem:no-triangle}
  A Burling graph contains no triangle.
\end{lemma}

\begin{proof}
  Suppose, for the sake of contradiction, that this is not true. Let $ G $ be an oriented graph derived from a Burling tree $ T $ whose underlying graph contains a triangle.  
  Up to symmetry, a triangle has only two possible orientations. By Lemma~\ref{lem:DG-no-cycle}, the triangle in $ G $ is not a directed cycle. Therefore, it is transitively oriented, i.e. it has an arc
  $uv$ and two arcs $uw$ and $wv$. So, $ v $ and $ w $ are both
  out-neighbors of $ u $, and thus are included in the a common branch of $ T $. Therefore, there is no arc between $ v $ and $ w $, a contradiction.
\end{proof}
	
The next lemma shows that all oriented Burling graphs can be derived
from Burling trees with specific properties. This will reduce
complication in some proofs.  Let $G$ be an oriented graph derived
from a Burling tree $T$.  An arc $uv$ of $G$ is a \emph{top-arc with respect to $T$} if $v$ is the out-neighbor of $u$ that is closest
(in $T$) to the root of $T$.  An arc $uv$ of $G$ is a \emph{bottom-arc with respect to $T$} if $v$ is the out-neighbor of $u$ that is
furthest (in $T$) from the root of $T$. Note that it is possible that
a top (resp.\ bottom) arc with respect to $T$ is no longer a top-arc
(resp.\ bottom) when $G$ is derived from another tree $T'$.  This is
why we add ``with respect to $T$'' but we omit it when $T$ is clear
from the context. Informally, top-arcs and bottom-arcs are arcs that can be subdivided while preserving being Burling, as we will show in Lemma~\ref{l:subdivAll}.

\begin{figure}
	\begin{center}
		\includegraphics[width=5.5cm]{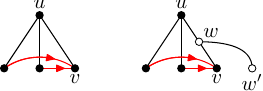}
		\caption{Transforming $v$ into a non-last-born.\label{f:subdiv-6}}
	\end{center}
\end{figure}

\begin{figure}
	\begin{center}
		\includegraphics[width=8cm]{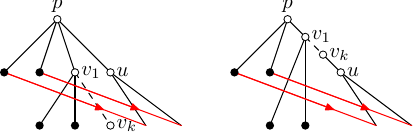}
		\caption{Transforming $v$ into a last-born.\label{f:subdiv-3}}
	\end{center}
\end{figure}

\begin{lemma}
	\label{l:chooseT}
	If $G$ is an oriented Burling graph derived from a Burling tree
	$(T, r, \lastBorn, \choosePath)$, then $G$ can be derived from a
	Burling tree $(T', r', \lastBorn', \choosePath')$ such that:
	$v\in V(G)$ if and only if $v$ is neither the root nor a last-born
	of $T'$. Moreover, for all arcs $uv$ of $G$, $uv$ is a top (resp.\
	bottom) arc of $G$ with respect to $T$ if and only if it is a top
	(resp.\ bottom) arc of $G$ with respect to~$T'$.
\end{lemma}

\begin{proof}
	We apply several transformations to $(T, r, \lastBorn, \choosePath)$
	in order to obtain a Burling tree
	$(T', r', \lastBorn', \choosePath')$ from which $G$ can be derived
	and that satisfies the conclusion.  We start with
	$(T', r', \lastBorn', \choosePath') = (T, r, \lastBorn,
	\choosePath)$.
	
	If $r\in V(G)$, then add a vertex $r'$ adjacent to $r$, and set
	$\lastBorn(r')=r$.
	
	If some vertex $v$ of $G$ is a last-born (possibly $r$ after the
	first step), then let $u$ be its parent. Delete the edge $uv$ from
	$T'$ and add a vertex $w$ adjacent to $u$ and $v$. Set
	$\lastBorn'(u) = w$.  Add a vertex $w'$ adjacent to $w$, and set
	$w'= \lastBorn(w)$.  For every vertex $x$ such that
	$v\in \choosePath(x)$, replace
	$\choosePath'(x)$ by $\choosePath'(x) \cup \{w\}$.  See
	Figure~\ref{f:subdiv-6}.  Note that all the new vertices added to
	$T'$ at this step are last-borns and are not in $G$. 
	
	
	Finally, suppose that some non-root vertex $ v $ of the tree is not a last-born of the tree and is not in $ V(G)$ neither. In that case, let $ p $ be the parent of $ v$ (which exists, because $v$ is not the root) and let $ u $ be the last-born of $ p $. So, $ u \neq v $. Let $ v_1 = v $ and for $ i \geq 1$, if $ v_i $ is not a leaf, let $ v_{i+1}  = \lastBorn(v_i)$. Let $ v_k $ be a leaf. Define a new Burling tree $(T', r, \lastBorn', \choosePath') $ as follows: remove the edge $pu$ and add the edge $v_ku$. This defines the tree $ T'$. Now, define $ \lastBorn'(p)=v_1$, $\lastBorn'(v_k) = u$, and for any other vertex $ w $, define $\lastBorn'(w) = \lastBorn(w)$. Moreover, define $ \choosePath'(v) = \varnothing$, for every vertex $ w \in V(T') \sm \{v\} $ with $ u \in \choosePath(w) $, define $ \choosePath'(w)= \choosePath(w) \cup \{v_1, \dots, v_k\}$, and for any other vertex $ w'$, define $ \choosePath'(w') = \choosePath(w)$. In the new Burling tree, $ v $ is a last-born. Notice that because of the previous transformations, none of the vertices $v_1, \dots, v_k $ are in $V(G)$, and therefore the adjacencies in $ G $ do not change. Moreover, this transformation does not cancel the effect of the previous ones. 
	
	Applying these transformations in the presented order leads to a Burling tree that satisfies the
	conclusions of the lemma and for every vertex $ x $, we have
	$\choosePath'(x)\cap V(G) = \choosePath(x)\cap V(G)$ for $x\in V(G)$, so that $G$ can be derived from
	$T'$. The orders in which the vertices of $G$ appear along branches
	of $T$ is the same as along the branches of $T'$, so top and bottom
	vertices are the same for the two representations.
\end{proof}

\subsection*{Subdivisions}

We here study how to obtain new oriented Burling graphs by subdividing
arcs.  Before starting the proofs, it is worth observing that in a
Burling tree $(T, r, \lastBorn, \choosePath)$, when $v$ is the
last-born of some vertex $u$, one may obtain another Burling tree
$(T', r, \lastBorn', \choosePath')$ by setting
$(T', r, \lastBorn', \choosePath') = (T, r, \lastBorn, \choosePath)$
and then applying the following transformations to $T'$.  First, delete
the edge $uv$ from $T'$ and add a new vertex $w$ adjacent to $u$ and
$v$.  Then set $\lastBorn'(u) = w$ and $\lastBorn'(w) = v$.  To define
$\choosePath'$, add $w$ to all sets $\choosePath(x)$ that contain $v$.
A fact that we do not state formally but is easy to check and is implicit in some of the proofs below is that this new Burling tree $T'$ is equivalent to
$T$ in the sense that every graph that can be derived from $T$ can be
derived from $T'$.  This is simply because
$V(G) \subseteq V(T) \subseteq V(T')$ and because for all $x\in V(G)$,
$\choosePath(x)\cap V(G) = \choosePath'(x)\cap V(G)$. 

\emph{Subdividing} an arc $uv$ into $uwv$ in an oriented graph means
removing the arc $uv$ and adding instead a directed path $uwv$ where $w$
is a new vertex.
	
\begin{figure}
  \begin{center}
    \includegraphics[width=6cm]{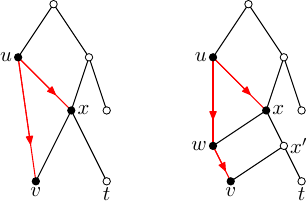}
    \caption{Subdividing a bottom-arc.\label{f:subdiv-9}}
  \end{center}
\end{figure}

\begin{lemma}
  \label{l:bottomSubdiv}
  Let $G$ be an oriented graph derived from a Burling tree $T$ and
  $uv$ be a bottom-arc of $G$.  The graph $G'$ obtained from $G$ by
  subdividing $uv$ into $uwv$ can be derived from a Burling tree $T'$
  in such a way that:
  \begin{itemize}
  	\item $uw$ is a bottom-arc of $G'$,
  	\item $wv$ is both a bottom-arc and a top-arc of $ G'$,
  	\item every top-arc of
  	$G$ with respect to $T$ (except $uv$) is a top-arc of $G'$ with
  	respect to $T'$,
  	\item every bottom-arc of $G$ with respect to $T$
  	(except $uv$) is a bottom-arc of $G'$ with respect to $T'$.
  \end{itemize}
\end{lemma}
	
\begin{proof}
  By Lemma~\ref{l:chooseT}, we may assume that $v$ is not a
  last-born. Let $x$ be the parent of $v$, and let $t$ be the last-born of
  $x$.  See Figure~\ref{f:subdiv-9}.

  Build from $T$ a Burling tree $(T', r', \lastBorn', \choosePath')$
  by removing the edge $xt$ from~$T$.  Then add to $T'$ a path $xx't$,
  and set $\lastBorn'(x)=x'$, and $\lastBorn'(x')=t$.  Add to $T'$ a
  new vertex $w$ adjacent to $x$.  Set
  $\choosePath'(u)=\{w\} \cup \choosePath(u) \sm \{v\}$.  Set
  $\choosePath'(w) = \{x', v\}$.  Replace $\choosePath'(z)$ by
  $\choosePath'(z) \cup \{x'\}$ for all $z\neq v$ such that
  $t\in \choosePath(z)$ or $v\in \choosePath(z)$.

  We see that the oriented graph $G'$ obtained from $G$ by subdividing
  arc $uv$ into $uwv$ can be derived from $T'$. 
\end{proof}

\emph{Top-subdividing} an arc $uv$ into $wu$ and $wv$ means removing
$uv$ and add instead two arcs $wv$ and $wu$ where $w$ is a new vertex.

\begin{figure}
  \begin{center}
    \includegraphics[width=9cm]{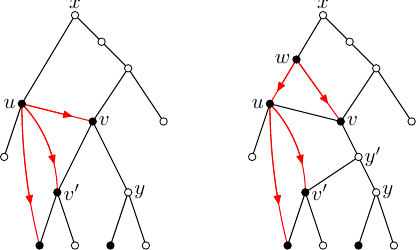}\\\rule{0em}{3ex}\\
    \includegraphics[width=9cm]{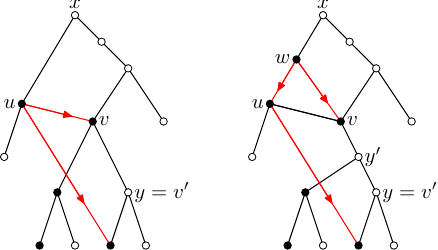}
    \caption{Top-subdividing a top-arc.\label{f:subdiv-12}}
  \end{center}
\end{figure}

\begin{lemma}
  \label{l:topSubdiv}
  Let $G$ be an oriented graph derived from a Burling tree $T$ and
  $uv$ be a top-arc of $G$ such that $u$ is a source of $G$.  The
  graph $G'$ obtained from $G$ by top-subdividing $uv$ can be derived
  from a Burling tree $T'$ in such a way that:
  \begin{itemize}
  	\item $wv$ is a top-arc, 
  	\item $wu$ is a bottom-arc,
  	\item every top-arc of $G$ with respect to $T$ (except
  	$uv$) is a top-arc of $G'$ with respect to $T'$,
  	\item every bottom-arc
  	of $G$ with respect to $T$ (except $uv$) is a bottom-arc of~$G'$
  	with respect to $T'$.
  \end{itemize} 
\end{lemma}
	
\begin{proof}
  Note that $u$ is not a last-born since there exists an arc $uv$ in
  $G$.  Let $x$ be the parent of $u$.  Let $y$ be the last-born of $v$
  (if $v$ is a leaf of $T$, just add $y$ to $T$). By
  Lemma~\ref{l:chooseT}, we may assume that $y$ is not in $G$.  Let
  $v'$ be the child of $v$ such that $v'\in c(u)$ (it is possible that $v'=y$ if
  no child of $v$ is in $c(u)$, but in that case we just add $y$ to $c(u)$).  See
  Figure~\ref{f:subdiv-12}, where the cases $v'\neq y$ and $v'=y$ are
  represented. Notice that the proof below applies to the two cases at the same time.

  Build from $T$ a Burling tree $(T', r', \lastBorn', \choosePath')$
  by removing the edges $xu$, $vv'$ and $vy$ from $T$.  Then add to
  $T'$ the edge $vu$, the path $vy'y$ and the edge $y'v'$, and set
  $\lastBorn'(v)=y'$ and $\lastBorn'(y')=y$.  Add to $T'$ a new vertex
  $w$ adjacent to $x$ (in $T'$).  Set
  $\choosePath'(u)=\choosePath(u) \sm V(P)$ where $P$ is the path of
  $T'$ from $x$ to $v$.  Set $\choosePath'(w) = \{u\} \cup V(P)\sm \{x\}$.
  Replace $\choosePath'(z)$ by $\choosePath'(z) \cup \{y'\}$ for all
  $z$ such that $y\in \choosePath(z)$ or $v'\in \choosePath(z)$.

  We see that the oriented graph $G'$ obtained from $G$ by top-subdividing
   $uv$ into $wu$ and $wv$ can be derived from $T'$ because $ v \notin \choosePath'(u) $ and $ u, v \in \choosePath'(w) $, and the rest of the arcs between the vertices of $ G $ have remained unchanged. Observe that $v$ and
  $u$ are the only vertices of $G$ in $V(P)\sm \{x\}$ since $uv$ is a top-arc of
  $G$. Observe that no vertex $z$ of $G$ has $u$ in $\choosePath(z)$
  since $u$ is a source of $G$. 
\end{proof}

\begin{lemma}
  \label{l:subdivAll}
  Let $G$ be an oriented Burling graph derived from a Burling
  tree~$T$. Any graph obtained from $G$ after performing the following
  operations is an oriented Burling graph:
  \begin{itemize}
  \item Replacing some bottom-arcs $uv$ by a path of length at
    least~1, directed from $u$ to $v$.
  \item Replacing some top-arcs $uv$ such that $u$ is a source of $G$
    by an arc $wv$ and a path of length at least~1 from $w$
    to $u$.
  \end{itemize}
\end{lemma}

\begin{proof}
  Clear by repeatedly applying Lemmas~\ref{l:bottomSubdiv}
  and~\ref{l:topSubdiv}. 
\end{proof}

\begin{figure}[p]
  \begin{center}
    \includegraphics[width=8cm]{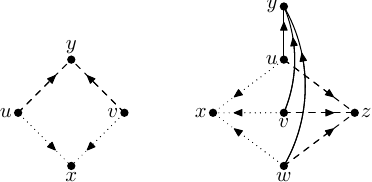}
    \caption{Some subdivisions of complete bipartite graphs as Burling graphs. \label{f:firstExample-7}}
  \end{center}
\end{figure}

\begin{figure}[p]
  \begin{center}
    \includegraphics[width=10cm]{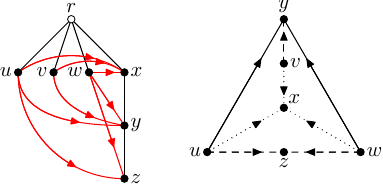}
    \caption{Some subdivisions of $K_4$ that are Burling graphs. \label{f:SK4}}
  \end{center}
\end{figure}
	
\begin{figure}[p]
  \begin{center}
    \includegraphics[width=12cm]{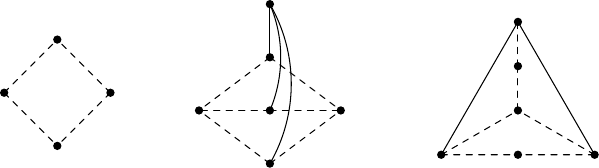}
    \caption{Some Burling graphs. Dashed edges can be subdivided.\label{f:firstExample-14}}
  \end{center}
\end{figure}

Let us now give applications of Lemma~\ref{l:subdivAll}.  
In Figures~\ref{f:firstExample-7} and~\ref{f:SK4}, some oriented Burling graphs are represented. In both figures, dotted arcs represent top-arcs that have a source as one of their end-points, dashed arcs represent bottom-arc, and all the other arcs are represented by solid arcs. Therefore, by Lemma~\ref{l:subdivAll}, by top-subdividing any of the dotted arcs and subdividing any of the dashed arcs, we obtain an oriented Burling graph. In Figure~\ref{f:firstExample-7}, the oriented graphs derived from the
Burling trees from Figure~\ref{f:square} are represented. In Figure~\ref{f:SK4}, a Burling tree,
together with the oriented graph derived from it, are represented. By considering its underlying graph, we see how to obtain several subdivisions
of $K_4$, namely any subdivision in which every edge except $ uy $ and $ wy $ is possibly subdivided. We will see later that this figure in fact provides all
subdivisions of $K_4$ that are Burling graphs (see
Lemma~\ref{l:K4}). As a consequence, all graphs arising from the three
non-oriented graphs in Figure~\ref{f:firstExample-14} by subdividing
dashed edges are Burling graphs.

\begin{lemma} \label{thm:contraction} Let $G$ be an oriented derived
  graph, and let $ uv $ be an arc such that $ N^+(u) = \{v\} $ and
  $ N^-(v) = \{u\}$. Then the graph $ G' $ obtained by contracting
  $ uv $ is also an oriented derived graph and the top-arcs (resp.\ bottom-arcs)
  of $ G $ except $ uv $ are the top-arcs (resp.\ bottom-arcs) of $ G'$.
\end{lemma}
	
\begin{proof} 
  Suppose that $G$ is derived from the Burling tree
  $ (T, r, \lastBorn, \choosePath) $. Let $S$ be the vertex-set
  (possibly empty) of the path starting at the last-born of the parent
  of $ u $ and ending at the parent of $ v $ in $ T $. Notice that
  $ S \subseteq \choosePath(u) $, and since
  $ \choosePath(u) \cap V(G) = \{v\}$, no vertex of $S$ and no
  descendant of $v$ in $c(u)$ is a vertex of $ G $.
		
  Define $ \choosePath'(u) = S \cup c(v) $ and define
  $ \choosePath'(w) = \choosePath(w) $ for any vertex $ w $ of $ T $
  other than $ u $.  It is easy to see that
  $ (T, r, \lastBorn, \choosePath') $ is a Burling tree. The graph
  $ G' $ is derived from this new Burling tree. Indeed, $ G' $ is the subgraph of the graph fully derived from $ (T, r, \lastBorn, \choosePath') $ induced on $ V(G) \sm \{v\} $.

  Finally, it is easy to see that no top-arcs or bottom-arcs are changed except for~$uv$.
\end{proof}

\section{$k$-Burling graphs}
\label{sec:kBurling}

An oriented graph $G$ is a \emph{oriented $k$-Burling graph} if it can
be derived from a Burling tree $T$ such that on each branch of $T$, at
most $k$ vertices belong to $G$. In such a case, we say that $ G $ is derived from $ T $ as a $ k$-Burling graph. Note that the empty graph is the
unique 0-Burling graph (in fact, the empty graph is $k$-Burling for all
integers $k\geq 0$).  In the next sections, 2-Burling graphs will be
useful and we need to describe them more precisely. Since it is not
much harder to describe $k$-Burling graphs in general, we do this here
by the mean of the so-called \emph{$k$-sequential graphs}.

The \emph{nobility} of an oriented graph is the smallest integer $k$
such that $G$ is a $k$-Burling graph. The \emph{nobility} of a
non-oriented Burling graph $G$ is the smallest nobility of an oriented
Burling graph $G'$ such that $G$ is the underlying graph of $G'$.   

\begin{figure}
  \centering
  \includegraphics[height=8cm]{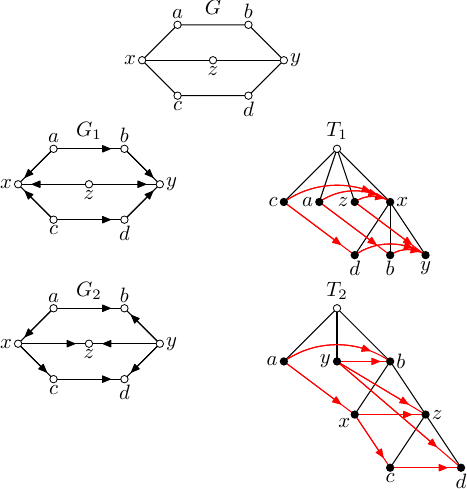}
  \caption{Nobility of a graph.\label{f:kBur}}
\end{figure}

Let us see examples. On Figure~\ref{f:kBur}, two oriented graphs $G_1$
and $G_2$ are represented.  Since $G_1$ can be derived from $T_1$, we
see that $G_1$ is a 2-Burling graph.  In fact, the nobility of $G_1$
must be 2, because since $c$ is a source of degree~2, its two
out-neighbors must be on the same branch.  Similarly, $G_2$ is a
3-Burling graph and has nobility~3 (because of $y$ being a source of
degree~3).  Hence, $G$ is a 2-Burling graph (in fact it has nobility 2
because we will soon see that only forests have nobility 1).  This
shows that an oriented graph (for instance $G_2$) may have a nobility
different from the nobility of its underlying graph.

\begin{figure}
  \centering
  \includegraphics[width=10cm]{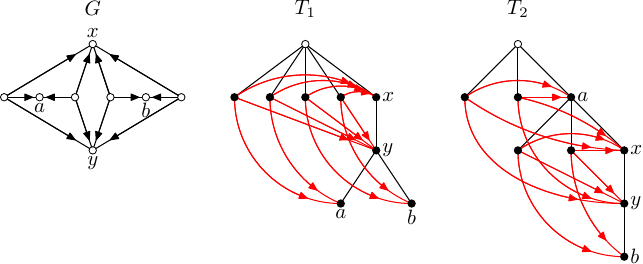}
  \caption{$G$ can be derived from $T_1$ and $T_2$.\label{f:kBurTwoRep-4}}
\end{figure}

In Figure~\ref{f:kBurTwoRep-4}, an oriented Burling graph $G$ with
nobility~3 is represented together with two Burling graphs $T_1$ and
$T_2$. Observe that $T_2$ has a branch that contains four vertices of
$G$. So, $ G $ is derived from $T_1 $ as a 3-Burling graph, and is derived from $ T_2 $ as a 4-Burling graph.

It should be pointed out that the nobility of an oriented graph may be
strictly greater than its maximum out-degree, as shown on
Figure~\ref{f:nobilityNotMd-3}.  The graph $G$ has three sources with
out-neighborhood $\{1, 2, 3\}$, $\{2, 3, 4\}$ and $\{3, 4, 5\}$. 
Assume that $ G $ is derived from a Burling tree $ T $. At least four vertices among 1, 2, 3, 4 and~5 must lie on the same branch of $T$. We only sketch the proof: on the branch of $ T $ that contains 1, 2, and 3, either 1 is the farthest vertex from the root or it is not. In the former case, if 4 is not a descendant of both 2 and 3, then 1, 2, 3, and 4 are all on a common branch, and if 4 is a descendant of both 2 and 3, then any branch containing 4, contains both 3, and 5. So, in particular, 2, 3, 4, and 5 are on a common branch. In the latter case (where 1 is not the farthest from root among 1, 2, and 3), any branch containing both 2 and 3 contains 1 as well. So, in particular 1, 2, 3, and 4 are on a common branch.
So, the nobility of $G$ is~4.

\begin{figure}
  \centering
  \includegraphics[height=4cm]{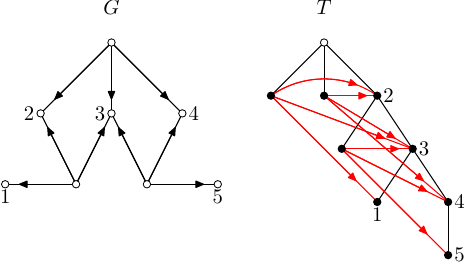}
  \caption{$G$ has maximum out-degree 3 and nobility 4.\label{f:nobilityNotMd-3}}
\end{figure}

\subsection*{1-Burling graphs}

An \emph{in-tree} is any oriented graph obtained from a rooted tree
$(T, r)$ by orienting every edge towards the root.  Formally, $e=uv$
is oriented from $u$ to $v$ if and only if $v$ is on the unique path of
$T$ from $u$ to $r$.  Notice that in an in-tree, every vertex but the
unique sink, has a unique out-neighbor. An \emph{in-forest} is an
oriented forest whose connected components are in-trees.

\begin{lemma}
  \label{l:1B}
  An oriented graph $G$ is an in-forest if and only if it is a
  1-Burling oriented graph.
\end{lemma}

\begin{proof}
  Suppose that $G$ is derived  from a Burling tree $T$ as a 1-Burling graph.  By
  Lemma~\ref{l:source}, every Burling graph contains a source and
  since the out-neighborhood of any vertex is included in a branch and
  $G$ is 1-Burling, this source has degree at most~1.  So, every
  1-Burling oriented graph has a source of degree at most~1.  This
  implies by an easy induction that every 1-Burling oriented graph
  is an in-forest.

  To prove the converse statement, it is enough to check that for
  every 1-Burling graph $G$ derived from a Burling tree
  $(T, r, \lastBorn, \choosePath)$ and every vertex $v$ of $G$, adding
  an in-neighbor $u$ of $v$ with degree~1 yields a 1-Burling graph
  $G'$.  Here is how to construct $G'$.  Build a rooted tree $T'$ from
  $T$ by adding a new root $r'$ adjacent to $r$.  Define for $V(T')$
  the functions $\lastBorn'$ and $\choosePath'$ as equal to
  $\lastBorn$ and $\choosePath$ for vertices of $T$.  Nominate $r$ as
  the last-born child of $r'$ and add $u$ as a non-last-born child of
  $r'$.  Then consider a branch $B$ of $T$ that contains $v$ and set
  $\choosePath'(u) = B$.  Note that by definition of 1-Burling graphs,
  $B\cap V(G) = \{v\}$.  So $G'$ is indeed derived from
  $(T', r', \lastBorn', \choosePath')$ and it is clearly a 1-Burling
  graph.
\end{proof}

\subsection*{Top-sets}

When $G$ is derived form a Burling tree
$(T, r, \lastBorn, \choosePath)$, we call the \emph{top-set} of $G$ the
set $S$ of all vertices $v$ of $G$ such that $v$ is the unique vertex
of $G$ in the branch of $T$ from $r$ to $v$.  See
Figure~\ref{f:seq-3}. The top-set of the graph presented in Figure~\ref{f:seq-3} is $ \{a,b,c,d,e\} $.

\begin{figure}
  \begin{center}
    \includegraphics[width=10cm]{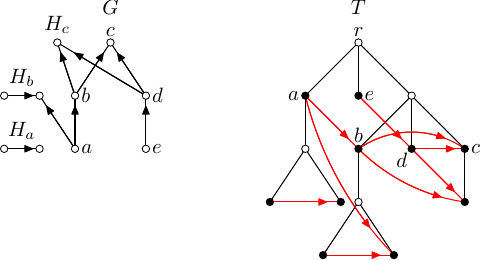}
    \caption{A Burling graph viewed as a sequential graph.\label{f:seq-3}}
  \end{center}
\end{figure}

\begin{lemma}
  \label{l:topForest}
  If $S$ is the top-set of the oriented graph $G$ derived from a
  Burling tree, then $G[S]$ is an in-forest. Moreover, if $G$ is a
  $k$-Burling graph ($k\geq 1$), then $G\sm S$ is a $(k-1)$-Burling
  graph.
\end{lemma}

\begin{proof}
  The graph $G[S]$ is clearly 1-Burling from the definition of the
  top-sets and also $G\sm S$ is a $(k-1)$-Burling graph.  So, $G[S]$ is
  an in-forest by Lemma~\ref{l:1B}. 
\end{proof}

  When  $G$ is derived from a Burling tree $T$, every vertex $u$
  of $G$ has a unique ancestor in the top-set of $G$.  This ancestor
  is called the \emph{top-ancestor} of $u$.

  \begin{lemma}
    \label{l:topAncestor}
    Let $G$ be derived from a Burling tree
    $(T, r, \lastBorn, \choosePath)$. Let $u$ and $v$ be two vertices of
    $G$ with top-ancestors $u'$ and $v'$, respectively. If $uv$ is an
    arc of $G$ then either:
    \begin{enumerate}[label=(\roman*)]
    \item $u'=v'$, $u'\neq u$, and $v' \neq v$, or
     \item $u=u'$ and $uv'\in A(G)$.
    \end{enumerate}
    \end{lemma}

\begin{proof}
  Suppose $u=u'$.  So, $v\in \choosePath(u)$.  The branch from $r$ to
  $v$ therefore contains $p(u)$, and since $u$ is in the top-set, $v'$
  must be in the branch from $p(u)$ to $v$ (and $v'\neq p(u)$).
  Hence, $v'\in \choosePath(u')$.  So, $u=u'$ and $uv'\in A(G)$.

  Suppose $u\neq u'$. So, $u'$ is ancestor of $p(u)$.  By
  Lemma~\ref{lem:simple}, $p(u)$ is an ancestor of $p(v)$, so $u'$ is
  an ancestor of $v$.  Hence $u'=v'$.  Also $v\neq v'$ because $u$ and
  $v'$ are in the same branch. 
\end{proof}

\subsection*{Sequential graphs}

Top-sets suggest defining Burling graphs as the graphs obtained from
the empty graph by repeatedly adding in-forests, with several
precise rules about the arcs between them.  A graph obtained after
$k$ steps of such a construction will be called a $k$-sequential
graph, and we will prove that $k$-sequential graphs are equivalent to
$k$-Burling graphs.  The advantage of $k$-sequential graphs is that
they have no shadow vertices like in the definition of derived
graphs. Also, they directly form a hereditary class, there is no need
to say that we take all induced sugraphs of something previously
defined by induction as in the definition of Burling graphs through
the Burling sequence. The price to pay for that it that we have to
maintain a set of stable sets in the inductive process.

Recall that in an in-forest, every vertex has at most one out-neighbor.
Also every connected component of an in-forest is an in-tree, and
therefore contains a unique sink.

The \emph{0-sequential graph} is the pair
$(G, \mathcal S)$ where $G$ is the empty graph (so
$V(G) = \varnothing$) and $\mathcal S = \{\varnothing\}$).    
For
$k\geq 1$, a $k$-sequential graph is any pair
$(G, \mathcal S)$, where $G$ is an oriented graph and $\mathcal S$ is a
set of stable sets of $G$, obtained as follows:

\begin{enumerate}[label=(\roman*)]
\item Pick a (possibly empty) in-forest $H$.
\item For every vertex $v$ of $H$, pick a  
  $(k-1)$-sequential graph $(H_v, \mathcal R_v)$.
\item For every vertex $u$ of $H$ that is not a sink, consider
  the unique out-neighbor $v$ of $u$, choose a stable set $R$ in
  $\mathcal R_v$ and add all possible arcs from $u$ to $R$.
\item The previous steps define all the vertices and arcs of $G$.
\item Set
  $\mathcal S = \{\varnothing\} \cup \left\{ \{v\} \cup R: v\in V(H), R\in
  \mathcal R_v\right\}$.
\end{enumerate}

An oriented graph $G$ is \emph{k-sequential} if, for some set $\mathcal
S$, $(G, \mathcal S)$ is $k$-sequential.  The in-forest $H$ in the
definition above is called the \emph{base forest} of the
$k$-sequential graph. Observe that the graph on
one vertex is $k$-sequential for all $k\geq 1$ and the empty graph is
$k$-sequential  for all $k\geq 0$. The graph $ G $ in Figure~\ref{f:seq-3} is a 2-sequential graph. The in-forest $ H $ is the subgraph of $ G$ induced by $ \{a,b,c,d,e\} $. The graphs $ H_a$, $H_b$, and $H_c $ are shown in the figure, and they are all 1-sequential graphs.

\begin{lemma}
  \label{l:kSeq}
  For all $k\geq 0$, an oriented graph $G$ is a $k$-Burling graph with
  top-set $S$ if and only if it is a $k$-sequential graph with base
  forest $H$ and $H=G[S]$.
\end{lemma}

\begin{proof}
  Let $G$ be a $k$-Burling graph derived from a Burling tree
  $(T, r, \lastBorn, \choosePath)$ as a $ k $-Burling graph, and let $S$ be the top-set of $G $. Our goal is to show that~$ G $ is $k$-sequential with base forest $H=G[S]$. Call 
  any branch of $T$ that contains $r$ or is
  empty a \emph{top-branch} of $T$.  We prove by induction on $k$ that $(G, \mathcal S)$ is a
  $k$-sequential graph where
  $$\mathcal S =\{\varnothing\} \cup \{S: \text{ $S$ is the
    intersection of a top-branch of $T$ with $V(G)$}\}.$$

  For $k=0$, this is clear. Suppose that $k\geq 1$ and that the statement holds
  for $k-1$.  Let us prove that $(G, \mathcal S)$ is a $k$-sequential
  graph, by building it as in the definition above.  Define $H=G[S]$
  where $S$ is the top-set of $G$.  By Lemma~\ref{l:topForest}, $H$ is
  an in-forest.

  For each vertex $v$ of $H$, consider the Burling tree
  $(T_v, v, \lastBorn_v, \choosePath_v)$ where $T_v$ is induced by all
  descendants of $v$ in $T$ and $\lastBorn_v$, $\choosePath_v$ are the
  restrictions to $V(T_v)$ of $\lastBorn_v$ and $\choosePath_v$
  respectively.  By the induction hypothesis, the subgraph of $ G $ which is derived from
  $(T_v, v, \lastBorn_v, \choosePath_v)$ is a $(k-1)$-sequential graph, and we denote it by
  $(H_v, \mathcal R_v)$ and
  $$\mathcal R_v =\{\varnothing\} \cup \{S: \text{ $S$ is the
    intersection of a top-branch of $T_v$ with $V(H_v)$}\}.$$

  By Lemma~\ref{l:topAncestor}, all arcs of $G$ are either arcs of
  $H$, or arcs of $H_v$ for some $v\in V(H)$, or arcs of the form $uw$
  where $w$ is a descendant of $v$ such $uv\in A(H)$.  It follows that
  $G$ can be obtained from $H$ and the $H_v$'s by adding for every arc
  $uv$ of $H$ all arcs of the form $uw$ where
  $w\in c(u) \cap (V(G) \sm \{v\})$.  It follows that for every vertex
  $u$ of $H$ that is not a top-vertex and the unique out-neighbor $v$
  of $u$, all possible arcs from $u$ to $R$ are added where
  $R = c(u) \cap (V(G) \sm \{v\}) \in \mathcal R_v$.  It follows that
  $(G, \mathcal S)$ is a $k$-sequential graph.

  Let us prove the converse statement. Consider a $k$-sequential
  graph $(G, \mathcal S)$ obtained as in the definition from a base
  forest $H$ and $(k-1)$-sequential graphs $(H_v, \mathcal R_v)$ for
  each $v\in V(H)$.  We have to prove that $G$ is a $k$-Burling graph and
  $H=G[S]$ where $S$ is the top-set of $G$.  We prove by induction on
  $k$ that $G$ can be derived from a Burling tree
  $(T, r, \lastBorn, \choosePath)$ as a $ k$-Burling graph such that
  $$\mathcal S= \{  S : \text{$S$ is the intersection of a top-branch of $T$ with
    $V(G)$}\}.$$

  For $k=0$, this is clear, so suppose $k\geq 1$ and the statement is
  true for $k-1$.  So, $G$ is obtained from $H$ as in the definition
  of $k$-sequential graphs.  By Lemma~\ref{l:1B}, $H$ is a 1-Burling
  graph derived from a tree $(T_H, r_H, \lastBorn_H, \choosePath_H)$.
  By the induction hypothesis, for every $v\in V(H)$, $H_v$ can be
  derived from a Burling tree $(T_v, r_v, \lastBorn_v, \choosePath_v)$ and
  $$\mathcal R_v= \{ S : \text{$S$ is the intersection of a top-branch of
    $T_v$ with $V(H_v)$}\}.$$

  We now build a Burling tree $(T, r, \lastBorn, \choosePath)$ from
  $T_H$ and the $T_v$'s.  For every $v\in V(H)$, we add an edge from
  $v$ to $r_v$.  This defines $T$.  We set $r = r_H$.  We define the
  last-borns in $T$ as inherited from the last-borns in $T_H$ and the
  $T_v$'s, and declare $r_v$ to be the last-born of $v$ (except if $v$
  is not a leaf of $T_H$, in which case it keeps its last-born). For
  every vertex $u$ of $H$ that is not a a sink, we consider the unique
  out-neighbor $v$ of $u$, and the chosen set $R$ in $\mathcal R_v$.
  We set $\choosePath (u) = \{v\} \cup R$.  For every vertex $u$ of
  $H$ that is a sink, we set $\choosePath (u) = \varnothing$.  For all
  other vertices, we define $\choosePath (u)$ as inherited from
  $\choosePath_H$ or $\choosePath_v$.  We that $G$ can be derived from
  $(T, r, \lastBorn, \choosePath)$ and
  
  \hfil $\mathcal S= \{  S : \text{$S$ is the intersection of a top-branch of $T$ with
    $V(G)$}\}.$
\end{proof}

\subsection*{Pivots and antennas}

Suppose that $G$ is derived form a Burling tree
$(T, r, \lastBorn, \choosePath)$ with top-set $S$, we call
a \emph{pivot} of $G$ any sink of $G[S]$ and an \emph{antenna} of $G$ any
source of $G[S]$.

\begin{lemma}
  \label{l:pivotSink}
  If $G$ is derived from a Burling tree
  $(T, r, \lastBorn, \choosePath)$, then every pivot of $G$ is a sink
  of $G$ and every antenna of $G$ is a source of $G$.
\end{lemma}

\begin{proof}
  Let $S$ be the top-set of $G$.  By Lemma~\ref{l:kSeq}, there exists
  $k$ such that $G$ is a $k$-sequential graph with base forest
  $H=G[S]$.  By Lemma~\ref{l:topAncestor}, if $u$ is a pivot of $G$,
  there cannot be an arc $uv$ in $G$, so $u$ is a sink of $G$.
  Similarly, if $v$
  is an antenna of $G$, an arc $wv$ would contradict
  Lemma~\ref{l:topAncestor} so $v$ is a source of $G$.
\end{proof}

\begin{lemma}
  \label{l:pivotConnected}
  If a connected oriented graph $G$ is derived from a Burling tree
  $(T, r, \lastBorn, \choosePath)$ with top-set $S$, then $G[S]$ is an
  in-tree (in particular $G$ has a unique pivot).  Moreover, no vertex
  of $G$ is a strict descendant of an antenna of $G$.
\end{lemma}

\begin{proof}
  By Lemma~\ref{l:kSeq}, there exists an integer $ k $ such that $G$ is a $k$-sequential
  graph with base forest $H=G[S]$.  For the sake of contradiction, suppose that $H$ is disconnected. Let $X$
  and $Y$ be two connected components of $H$. By the definition of
  $k$-sequential graphs, $X$ and $Y$ are in distinct
  components of $G$, a contradiction to $G$ being connected. So, $ H $ is connected and thus is an in-tree.

  Again, for the sake of contradiction, let $ u $ be a strict descendant of an antenna of $G $. By the construction of $k$-sequential graphs, $ u $ and the unique pivot of $ G $ are in distinct connected components of $G$, a contradiction to $ G $ being connected. Therefore, no vertex of $G$ is a strict descendant of
  an antenna of $G$.
\end{proof}

The following lemma gives more properties of the top-set under stronger
connectivity assumptions.

An \emph{in-star} is an in-tree whose unique sink is adjacent to all other vertices. 

\begin{lemma}
  \label{l:pivotNoCutV}
  Let $G$ be a connected oriented graph derived from a Burling tree
  $(T, r, \lastBorn, \choosePath)$.  Suppose that $G$ has no
  cut-vertex and no vertex of degree at most~1. Then the following
  statements hold:
  \begin{enumerate}[label=(\roman*)]
    \item
      The top-set of $G$ is an in-star $S$ with at least two leaves
      (so $G$ has a unique pivot and its in-neighbors are the antennas
      of $G$).
    \item 
     All vertices of $S\sm \{v\}$ are sources of $G$ where $ v $ is the unique sink of~$ S $.
   \item The pivot of $G$ is an ancestor of all vertices of
     $V(G) \sm S$.
   \end{enumerate}
\end{lemma}

\begin{proof}
  Let $v$ be the sink of $G[S]$. If $G[S]$ is not an in-star, then
  there exists a directed path $yxv$ in $G[S]$.  By the construction
  of $k$-sequential graphs, we see that $x$ is a cut-vertex of $G$
  that separates $v$ from $y$. A contradiction.  
  By Lemma~\ref{l:pivotSink}, all vertices of $S\sm \{v\}$ are
  sources of $G$ because they are the antennas of $G$.  Since $G$ is
  connected, the antennas have no strict descendants, 
  by Lemma~\ref{l:pivotConnected}, the pivot of
  $G$ is an ancestor of all vertices of~$V(G) \sm S$.
\end{proof}

\subsection*{2-Burling graphs}

A leaf in an in-tree is a vertex with no in-neighbors (so the root is not a leaf unless the in-tree has only one vertex).
An \emph{oriented chandelier}\index{oriented chandelier} is any oriented graph $G$ obtained from an in-tree $G'$ whose root is of degree at least~2 by adding a vertex $v$ and all arcs $uv$ where $u$ is a leaf of~$G'$.  


\begin{lemma}
  \label{l:2burling}
  An oriented graph is an oriented chandelier if and only if it is a
  connected 2-Burling graph with no cut-vertex and no vertex of degree
  at most 1.
\end{lemma}

\begin{proof}
	First, suppose that $G$ is an oriented chandelier.  So $G$ is clearly connected, has no cut-vertex, and has no vertex of degree at most~1.  It remains to prove that it is a 2-Burling graph. Let $G'$ and $v$ be as in the definition of oriented chandelier. Let $u_1, \dots, u_k$ be the leaves of $G'$, and for $ i \in [k]$, let $ v_i $ be the neighbor of $ u_i $ in $ G $ ($v_i$'s are not necessarily distinct). Set $ G'' = G' \sm \{u_1, \dots, u_k\} $. Since in $G'$, the root has degree at least 2, $G''$ contains the root, and thus is a non-empty in-tree. By Lemma~\ref{l:1B}, $G''$ can be derived from a Burling tree $(T, r, \lastBorn, \choosePath)$ as a 1-Burling graphs (i.e.\ on every branch of $ T $, at most one vertex belongs to~$ V(G)$). Let us build a tree $T'$ from $T$. Add a new root $r'$ adjacent to $r$, and add $k $ new children $ v_1, \dots, v_k $ to $ r'$. This defines the rooted tree $ (T',r')$. Then, define $ \lastBorn'(r') = r $ and for any vertex $ x \in V(T') \sm \{r', v_1, \dots, v_k\}$, we define $ \lastBorn'(x) = \lastBorn(x) $. Notice that the vertices $v_1, \dots, v_k$ are leaves in $ T' $, thus $\lastBorn'$ is not defined for them. Now, for every $ i \in [k]$, let $ B_i $ be the branch of $ T $ starting at $ r $ and ending at $ u_i$. Define $\choosePath'(v_i) = B_i$ and $\choosePath'(r') = \varnothing$. For every vertex $ x \in V(T') \sm \{r', v_1, \dots, v_k\}$, set $ \choosePath'(x) = \choosePath(x)$. The tuple $(T', r', \lastBorn', \choosePath')$ is a Burling tree. Renaming $ r $ as $ v $, we see that $G$ can be derived from $T'$. Indeed, $ G $ is the subgraph of the oriented graph fully derived from $ T'$ induced by $V(G'') \cup \{v, v_1, \dots, v_k\}$. Moreover, on each branch of $ T'$, at most 2 vertices are in $ V(G)$, thus $G$ is a 2-Burling graph.
	
	Conversely, suppose that $G$ is a connected graph with no cut-vertex and no vertex of degree at most 1 that is derived as a 2-Burling graph from a Burling tree $ T $. By \Cref{l:pivotNoCutV}, $G$ has a unique pivot $v$, all antennas of $G$ are in-neighbors of $v$, and the rest of the vertices of $ G $ are all descendants of $ v $ in $ T $. In particular, considering $ v $ as a shadow vertex of $ T $, we see that $ G\sm v $ is a 1-Burling graph. Therefore, by Lemma~\ref{l:1B}, $ G\sm v $ is an in-forest. On the other hand, since $ G $ has no vertex cut, $ G \sm v $ is connected and thus is an in-tree. Let $ r $ be the root of this in-tree. Since $ r $ cannot be of degree 1 in $ G $, it has at least two in-neighbors. But $ v$ is not an in-neighbor of $r$ (because it is among its ancestor). Therefore, in $ G\sm v $, the root $ r $ has at least 2 children. Moreover, if a leaf $u$ of $G \sm v$ is not adjacent to $v$ in $G$, then, $u$ has degree at most~1 in $G$, a contradiction.  So, $v$ is adjacent to all leaves of $G'$. Hence $G$ is an oriented chandelier.
\end{proof}

In the construction of oriented chandeliers in the proof above, $v$ is
the pivot of $G$ and its neighbors are the antennas. The unique sink
of $G\sm v$ is called the \emph{bottom} of $G$.  Note that every
source of $G$ is an antenna. The pivot and the bottom are the only
sinks of $G$.  Also, in the Burling tree $T$ from which $G$ is
derived, every vertex of $G$ except the antennas are descendants of the
pivot.

\section{Star cutsets}
\label{sec:starCutsets}

In this section, we study star cutsets in derived graphs.

\begin{lemma}
  \label{lem:common-nghbr-branch}
  Suppose that $G$ is an oriented graph derived from a Burling tree
  $T$. Let $v$ and $w$ be two vertices of $G$ such that $v$ is an
  ancestor of $w$ in $T$. Then every neighbor of $w$ in $G$ is either
  an in-neighbor of both $v$ and $w$ or a descendant of~$v$.
\end{lemma}
		
\begin{proof}
  Let $u$ be a neighbor of $ w $ in $ G $. If $ u $ is an
  out-neighbor of $w$, then $p(w)$ is a ancestor of $ u $. However,
  $ p(w) $ is a descendant of $ v $ (possibly $ v $ itself). So $ u $
  is a descendant of $ v $. If $ u $ is an in-neighbor of $ w $, then
  $p(u)$ is an ancestor of $ w $, and therefore it is on the unique
  branch in $T$ between $ w $ and the root. This branch includes $ v $
  as well. There are two cases: either $ p(u) $ is a descendant of $ v $
  or $p(u)$ is an ancestor of $ v $. In the former case, $ u $ is a descendant of $ v $. In the latter case 
$ u $ is an in-neighbor of $ v $ because $ u $ is connected to every vertex in the path between
  $ w $ and the last-born of $ p(u) $, and $ v $ is on this path.
\end{proof}

\begin{lemma}
  \label{lem:pre-fullstarcutsetlemma2}
  Suppose that $ G $ is an oriented graph derived from a Burling tree
  $T$. Let $u$, $v$, and $w$ be three vertices of $ G $ such that $w$
  is a descendant of $v$ and $u$ is not a descendant of $v$. Then
  every path (not necessarily directed) in $G$ between $u$ and $w$
  contains an in-neighbor of $v$ in $G$.
\end{lemma}

\begin{proof}
  Let $P$ be a path in $G$ from $u$ to $w$.  Since $u$ is not a
  descendant of $v$ while $w$ is, $P$ must contain an edge $u'w'$ such
  that $u'$ is not a descendant of $v$ while $w'$ is.  By
  Lemma~\ref{lem:common-nghbr-branch} applied to $v$ and $w'$,
  $u'$ is a in-neighbor of $ v $.
\end{proof}

A \emph{full in-star cutset} in an oriented graph $G$ is a set
$ S = N^-[v] $ for some vertex $ v \in V(G)$ such that $G\sm S$ is
disconnected.  A \emph{full star cutset} in a graph $G$ (oriented or
not) is a set $S = N[v] $ for some vertex $ v \in V(G)$ such that
$G\sm S$ is disconnected.  A \emph{star cutset} in a graph $G$ (oriented or
not) is a set $S$ such that for some vertex $ v \in V(G)$, $\{v\} \subseteq S \subseteq  N[v]$, and $G\sm S$ is
disconnected. In this case, we say that the star cutset $ S $ is centered at $ v $.

We say that in graph $ G $, the star cutset $ S $ separates two vertices $ u $ and~$ v $ if $ u $ and $ v $ are in two distinct connected components of $ G \sm S$.

\begin{lemma}
  \label{lem:pre-fullstarcutsetlemma3}
  Suppose that $ G $ is an oriented graph derived from a Burling tree
  $T$. Let $u$, $v$, and $w$ be three vertices of $ G $ appearing in
  this order along a branch of $T$. Then every path (not necessarily
  directed) in $G$ from $u$ to $w$ goes through an in-neighbor of $v$
  in $G$.

  In particular, $N^-[v]$ is a full in-star cutset of $G$ and $N[v]$
  is a full star cutset of $G$, which separated $ u $ and $ v$.
\end{lemma}

\begin{proof}
  Follows from Lemma~\ref{lem:pre-fullstarcutsetlemma2}, and
  $u$ and $w$ are in distinct connected components of $G\sm N^-[v]$.
\end{proof}

\begin{lemma}
  \label{l:degA1}
  If a triangle-free oriented graph $G$ has a cut-vertex, then either
  $G$ has a full in-star cutset, or $G$ has a vertex of degree at
  most~1.
\end{lemma}

\begin{proof}
  Let $v$ be a cut-vertex of $G$.  Let $A$ and $B$ be two connected
  components of $G\sm v$.  If $|A|\leq 1$ or $|B|\leq 1$, then $G$ has
  a vertex of degree at most~1, so let us assume that $|A|\geq 2$ and
  $|B|\geq 2$.  Since $G$ is triangle-free, $A$ (resp.\ $B$) contains
  a non-neighbor $a$ (resp.\ $b$) of $v$.  It follows that $a$ and $b$
  are in distinct connected components of $G\sm N^-[v]$.  So, $G$ has
  a full in-star cutset centered at $v$.
\end{proof}

\begin{theorem}
  \label{th:decompD}
  If $G$ is an oriented Burling graph, then either $G$ has a full in-star
  cutset, or $G$ is an oriented chandelier, or $G$ contains a vertex of
  degree at most~1.
\end{theorem}

\begin{proof}
  Suppose that $G$ has no vertex of degree at most~1. In particular,
  $|V(G)|\geq 3$.  By Lemma~\ref{l:degA1}, we may assume that $G$ has
  no cut-vertex (in particular $G$ is connected since $|V(G)|\geq 3$).
  We may assume that $G$ is a 2-Burling graph, for otherwise some
  branch of $T$ contains at least three vertices of $G$ and by
  Lemma~\ref{lem:pre-fullstarcutsetlemma2}, $G$ has a full in-star
  cutset.  So, $G$ is a connected 2-Burling graph with no cut-vertex and no vertex of degree at most 1.  Hence by Lemma~\ref{l:2burling},
  $G$ is an oriented chandelier.
\end{proof}

Theorem~\ref{th:decompD} is best possible in the sense that every oriented chandelier has no full in-star
cutset.

\subsection*{The non-oriented case}

 We may prove a theorem similar to
Theorem~\ref{th:decompD} for non-oriented graphs with the same
method.  We do not present its proof, because it
was already obtained in~\cite{Chalopin2014} for a superclass of Burling graphs, the
so-called \emph{restricted frame graphs}. (For definition of restricted frame graphs, see Definition~2.2 of~\cite{Chalopin2014}.)

A non-oriented graph obtained from a tree $H$ by adding a vertex $v$
adjacent to every leaf of $H$ is called in~\cite{Chalopin2014} a
\emph{chandelier}. If the tree $G$ has the property that the neighbor
of each leaf has degree two, then the chandelier is a \emph{luxury
  chandelier}. Observe that if $G$ is an oriented chandelier, then $G$
has no full in-star cutset, but it may have a full
star cutset.  It can be proved that the luxury chandeliers are exactly
the graphs with no full star cutset that are also underlying graphs of
oriented chandeliers. The following theorem is from \cite{Chalopin2014}, Corollary 3.4.

\begin{theorem}[Chalopin, Esperet, Li and Ossona de Mendez]
  \label{th:chalop}
  If $G$ is a non-oriented connected Burling graph, then $G$ has a
  full star cutset, or $G$ is luxury chandelier, or $G$ is an induced
  subgraph of $P_4$.
\end{theorem}

\begin{figure}
  \begin{center}
    \includegraphics[width=8cm]{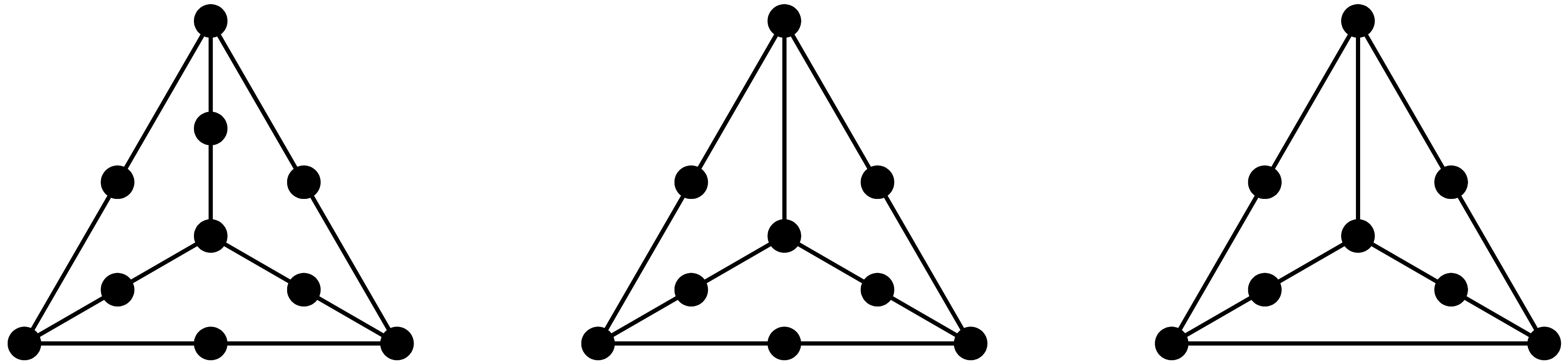}
    \caption{Some subdivisions of $K_4$ that are not
      Burling.\label{f:nonBurlingK4-4}}
  \end{center}
\end{figure}

As explained in~\cite{Chalopin2014}, Theorem~\ref{th:chalop} conveniently gives
graphs that are not Burling.  For instance, the three graphs
represented in Figure~\ref{f:nonBurlingK4-4} are not Burling because
they have no full star cutset and are not luxury chandelier (because
in a luxury chandelier, there exists a vertex that is contained in all
cycles).

In~\cite{CPST:subst}, the following question is asked: is there a constant $c$ such
that if a graph $G$ is triangle-free and all induced subgraphs of $G$
either are 3-colorable or admit a star cutset, then $G$ is
$c$-colorable?  It was seemingly never noticed that
Theorem~\ref{th:chalop} answers the question in the negative, since
luxury chandelier are 3-colorable and Burling graphs are triangle-free
graphs of unbounded chromatic number.

\section{Holes in Burling graphs}
\label{sec:holes}

A \emph{hole} in a graph $ G $ is a chordless cycle of length at
least~4.  We call hole of an oriented graph any hole of its underlying
graph. By Theorem~\ref{th:decompD}, since a hole has no in-star cutset
(whatever the orientation) and no vertex of degree~1, every hole in an
oriented graph derived from a Burling tree $T$ is an oriented
chandelier.  In particular, the explanations given after the proof of
Lemma~\ref{l:2burling} apply. Therefore, every hole $H$ has four
\emph{special} vertices that we describe here:

\begin{itemize}
  \item two sources called the
  \emph{antennas}, 
\item one common neighbor of the antennas that is also an ancestor
  in $T$ of all the vertices but the antennas, called the
  \emph{pivot},

  \item one sink distinct from the pivot, called the \emph{bottom}.
\end{itemize}

Every other vertex of $ H $ lies on a directed paths from an antenna to the bottom.
We call \emph{subordinate} vertex of a hole any vertex
distinct from its pivot and antennas (in particular, the bottom is
subordinate and is therefore a descendant of the pivot).

\begin{lemma}
  \label{l:compHole}
  Let $H$ be a hole in an oriented graph $G$ derived from a Burling
  tree $T$. Let $p$ be the pivot of $H$ and $C$ be the connected
  component of $G\sm N^-[p]$ that contains $H\sm N^-[p]$.  Then every
  vertex of $C$ is a descendant of $p$. 
\end{lemma}

\begin{proof}
  Suppose for the sake of contradiction, that the statement does not hold. So, $C$ contains a vertex $u$ that is not a descendant of
  $p$. Since every vertex of $ H \sm N^-[p]$ is a descendant of $p$,
  there exists a descendant $v$ of $p$ in $C$.  Let $P$ be a path from
  $u$ to $v$ in $C$. By
  Lemma~\ref{lem:pre-fullstarcutsetlemma2}, $P$ contains an
  in-neighbor of $p$.  This contradicts the definition of $C$. 
\end{proof}

A \emph{dumbbell} is a graph made of path $P=x\dots x'$ (possibly
$x=x'$), a hole $H$ that goes through $x$ and a hole $H'$ that
goes through $x'$.  Moreover $V(H) \cap V(P) = \{x\}$,
$V(H) \cap V(P') = \{x'\}$,
$V(H) \cap V(H') = \{x\} \cap \{x'\}$ and there are no other
edges than the edges of the path and the edges of the holes.

\begin{lemma}
  \label{l:dumbbell}
  Suppose a dumbbell with holes $H$, $H'$ and path $P=x \dots x'$ as
  in the definition is the underlying graph of some oriented graph $G$
  derived from a Burling tree $T$. Then either $x$ is not a
  subordinate vertex of $ H $ or $x'$ is not a subordinate vertex of
  $H'$.
\end{lemma}

\begin{proof}
  Suppose for the sake of contradiction, that the statement does not hold. So, the pivot $p$ of $H$ is in the interior of the path
  $H\sm x$ and the pivot $p'$ of $H'$ is in the interior of the path
  $H'\sm x'$.  By Lemma~\ref{l:compHole} applied to $H$, every vertex
  of $G \sm N[p]$ is a descendant of $p$ in $T$ (notice that
  $N[p] = N^-[p]$ since the pivot is a sink). By
  Lemma~\ref{l:compHole} applied to $H'$, every vertex of $G\sm N[p']$
  is a descendant of $p'$ in $ T$. It follows that $p$ and $p'$ are on
  the same branch of $T$, so up to symmmetry, we may assume that $p$
  is an ancestor of $p'$.  Let $q$ and $r$ be vertices of $H$ such
  that $p$, $q$ and $r$ are consecutive along $H$.  So, $q$ is an
  antenna of $H$ (because it is adjacent to the pivot), and $r$ is a
  descendant of $p$, but also of $p'$.  Thus $p'$ is between $p$ and
  $r$ in some branch of $T$. Now because $p$ and $r$ are both in
  $c(q)$, so is $p'$. Hence $q$ is adjacent to $p'$, a contradiction
  to the definition of dumbbells.
\end{proof}

A \emph{domino} is a graph made of one edge $xy$ and two holes $H_1$ and
$H_2$ that both go through $xy$. Moreover 
$V(H_1) \cap V(H_2) = \{x, y\}$ and there are no other
edges than the edges of the holes.

\begin{lemma}
  \label{l:domino}
  Suppose a domino with holes $H_1$, $H_2$ and edge $xy$
  as in the definition is the underlying graph of some oriented graph
  $G$ derived from a Burling tree $ T$. Then for some $z\in \{x, y\}$ and
  some $i\in \{1, 2\}$, $z$ is the pivot of $H_i$ and $z$ is a subordinate
  vertex of $H_{3-i}$.
\end{lemma}

\begin{proof}
  Let us first prove that one of $x$ or $y$ is the pivot of one of
  $H_1$ or $H_2$. 
  Otherwise, the pivot $p_1$ of $H_1$ is in the path $H_1\sm \{x, y\}$ and
  the pivot $p_2$ of $H_2$ is in the path $H_2\sm \{x, y\}$.  Suppose
  up to symmetry that $yx$ is an arc of $G$.  It follows that
  $x\in V(G) \sm (N[p_1] \cup N[p_2])$.  Because $x\neq p_1, p_2$ by
  assumption, and $x\notin N(p_1) \cup N(p_2)$ because in a hole, the neighbors of the
  pivot are sources.  By Lemma~\ref{l:compHole} applied to $H_1$ and to
  $H_2$, $x$ is a descendant of both $p_1$ and $p_2$.  It follows that up
  to symmetry, we may assume that $p_1$ is a descendant of $p_2$.

  Let $a$ and $a'$ be the antennas of $H_2$.  Note that $a, a' \neq x$.
  Up to symmetry, suppose that $x$, $a$, $p_2$ and $a'$ appear in this
  order along $H_2$.  Let $x'$ be the neighbor of $a$ in $H\sm p_2$
  (possibly $x=x'$).  Since $x$ and $x'$ are in the same component of
  $G\sm (N[p_1] \cup N[p_2])$, $x'$ is a descendant of both $p_1$ and $p_2$.
  And since $ax'\in A(G)$, we have $x'\in c(a)$, so $p_1\in c(a)$ and
  $ap_1\in A(G)$, a contradiction to the definition of dominos.

  We proved one of $x$ or $y$ is the pivot of one of $H_1$ or $H_2$.
  Up to symmetry, suppose that $x$ is the pivot of $H_1$.  It remains
  to prove that $x$ is a subordinate vertex of $H_2$.  First, $x$
  cannot be an antenna of $H_2$ because $yx\in A(G)$.  Hence, we just
  have to prove that $x$ being the pivot of $H_2$ yields a
  contradiction. So, suppose that $x$ is the pivot of $H_2$.  It
  follows that $y$ is an antenna of both $H_1$ of $H_2$, so it is a
  source of $G$.  Let $y_1$ and $y_2$ be the neighbors of $y$ in
  $H_1\sm x$ and $H_2\sm x$ respectively.  Vertices $x$, $y_1$ and
  $y_2$ are on the same branch $B$ of $T$ (because they are all in
  $c(y)$).  So, by Lemma~\ref{lem:pre-fullstarcutsetlemma2}, $y_1$,
  $x$ and $y_2$ appear either in this order or in the reverse order
  along $B$, because $y_1$ and $y_2$ are not centers of star cutsets of
  $G$.  If $ y_2 $ is the deepest vertex in $ T $ among the three,
  then $y_1$ is an ancestor of $x$ while being in the hole $H_1 $ for
  which $x$ is the pivot, a contradiction. On the other hand, if
  $y_1 $ is the deepest, then $ y_2 $ is an ancestor of $x$ while
  being in the hole $H_2 $ for which $x$ is the pivot, again a
  contradiction.
\end{proof}

A \emph{theta} is a graph made of three internally vertex-disjoint
paths of length at least~2, each linking two vertices $u$ and $v$ called
the \emph{apexes} of the theta (and such that there are no other edges
than those of the paths).  A \emph{long theta} is a theta such that
all the paths between the two apexes of the theta have length at
least~3.

\begin{lemma} \label{lem:theta-pivot} Suppose a long theta with apexes
  $ u $ and $ v $ is the underlying graph of some oriented graph $G$
  derived from a Burling tree $ T$. Then exactly one of $ u $ and
  $ v $ is the pivot of every hole of~$G$.
\end{lemma}

\begin{proof}
  Consider the three paths between $ u $ and $ v $, and let $ Q_1$,
  $Q_2 $, and $Q_3$ denote the set of the internal vertices of these
  three paths respectively (so,~$ |Q_i| \geq 2 $). For $i=1, 2, 3$,
  let $ H_i $ be the hole induced by $Q_i \cup Q_{i+1} \cup \{u, v\}$
  (with subscript taken modulo~3) and let $ p_i$, $a_i$, and $ a'_i $
  denote the pivot and the two antennas of $ H_i $.
		
  For the sake of contradiction, assume that there is a hole in $ G $,
  say $ H_1 $, for which neither of $ u $ and $ v $ is a pivot.  So,
  without loss of generality, assume that $p_1 \in Q_1$.  Also, notice
  that $ a_1 $ and $a'_1 $ are the two neighbors of $p_1 $. Thus:
  \begin{enumerate}
  \item[(i)] neither of $ a_1 $ and $a'_1 $ are in $Q_2 $, and
    consequently, no vertex of $ Q_2 $ is a source in $ G $,
			
  \item[(ii)] because the underlying graph is a long theta, at least
    one of $ a_1 $ and~$a'_1 $, say $ a_1 $, is in $Q_1$.
  \end{enumerate}
		
  Now consider the hole $ H_2 $. Since the theta is long, if $ p_2 $
  is in $ Q_2 \cup \{u,v\} $, then at least one antenna of $ H_2 $
  must be in $ Q_2 $ which contradicts (i). Thus, $ p_2 \in
  Q_3$. Therefore, with the same argument as before, at least one of
  $ a_2 $ and $ a'_2 $, say~$ a_2 $ also should be in $ Q_3 $.
		
  Finally, consider the hole $ H_3 $. Notice that $ a_1, p_1 \in Q_1 $
  respectively form a source and a sink for $H_3 $. On the other hand,
  $a_2, p_2 \in Q_3 $ also, respectively form a source and a sink for
  $H_3 $. So, these are the four extrema of $ H_3 $. But then at
  least~three of these four vertices should be consecutive, which is
  impossible. This contradiction finishes the proof of the lemma.
\end{proof}

\section{New examples of non-Burling graphs}
\label{sec:examples}

We are now able to describe non-Burling triangle-free graphs that do
have full star cutsets, so going beyond the method following
from~\cite{Chalopin2014}, where all examples have no star cutsets.

\begin{figure}
  \begin{center}
    \includegraphics[width=6.7cm]{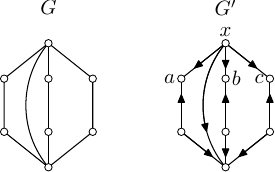}
    \caption{Theta+ and an orientation of it.\label{f:feedback-4}}
  \end{center}
\end{figure}

We call the (non-oriented) graph represented in Figure~\ref{f:feedback-4} (left) \emph{Theta+}. 

\begin{theorem}
	Theta+ is not a Burling graph. 
\end{theorem}

\begin{proof}
	For the sake of contradiction, suppose that $G$ is a Burling graph. So, some orientation of $ G $ can be derived from a Burling tree.  Hence, every $C_4$ of this orientation must contain a pivot, a bottom and two antennas.  One can check that with this condition, up to symmetry, the orientation of $ G $ is as $ G' $ shown in Figure~\ref{f:feedback-4}, right. Note that $a$, $b$ and $c$ are out-neighbors of $x$, so they must be on the same branch of the Burling tree. Therefore, by Lemma~\ref{lem:pre-fullstarcutsetlemma2}, one of $a$, $b$ or $c$ must be the center of a full in-star cutset, a contradiction. 
\end{proof}
 
Notice that
$G$ contains a vertex whose removal yields a tree.  Also, it admits
an orientation that is good for every hole (See Figure~\ref{f:feedback-4}, right).

\begin{figure}
  \begin{center}
    \includegraphics[width=10cm]{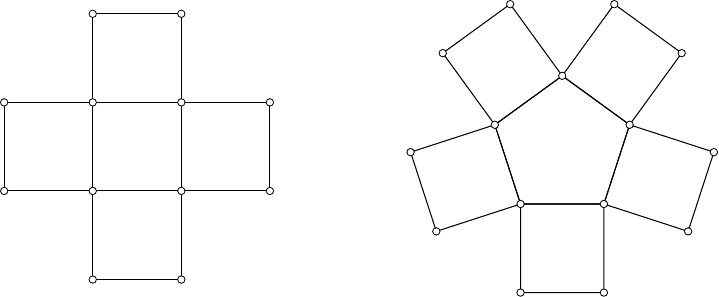}
    \caption{Examples of flowers.\label{f:flowers-3}}
  \end{center}
\end{figure}

A \emph{flower} is a graph $G$ made of a hole $H$ where every
edge $e$ is part of a hole $H_e$.  Moreover, $V(H) \cap V(H_e) = e$,
for all edges $e, f$ of $H$, $V(H_e) \cap V(H_f) = e \cap f$, and the
only edges and vertices of $G$ are those of the $H_e$'s.
In Figure~\ref{f:flowers-3}, two examples of flowers are represented. 

\begin{theorem}
  No flower is a Burling graph. 
\end{theorem}

\begin{proof}
  Suppose $G$ is a flower with a hole $H$ as in the definition.  Let
  $v$ be the pivot of $H$, and $u, w$ the two neighbors of $v$ in
  $H$.   So, $H_{uv}$ and $H$ form a domino, and by
  Lemma~\ref{l:domino}, one of the two vertices $ u $ and $ v $ should be the pivot of one of the two holes, and a subordinate vertex of the other. Notice that $ u $ cannot be a pivot of any of the two holes because $ uv $ is an arc. So, 
  $v$ is a subordinate vertex of $H_{uv}$.
  Similarly, $v$ is a subordinate vertex of $H_{vw}$.  Hence, $H_{vw}$
  and $H_{uv}$ contradict Lemma~\ref{l:dumbbell} (since $H_{vw}$
  and $H_{uv}$ form a dumbbell). 
\end{proof}

A \emph{wheel} is a graph made of hole $H$ called the \emph{rim}
together with a vertex $c$ called the \emph{center} that has at least
three neighbors in $H$.  Wheels are restricted frame graphs (see
Theorem A.1.\ and Figure 7 in \cite{Chalopin2014}). 
It was
claimed by Scott and Seymour in private communication that wheels are
not derived graphs, and independently, Davies also proved it recently. The first written proof of the theorem that we are aware of is in the master's thesis of the first author, see \cite{report}.

\begin{theorem}[Scott and Seymour~\cite{ScottSeymour17}, Pournajafi~\cite{report},
  Davies~\cite{davies2021trianglefree}]
  \label{thm:Wheel-not-Burling}
  No wheel is a Burling graph. 
\end{theorem}
	
\begin{proof}
  Suppose that a graph $G$ is wheel with rim $H$ and center $c$.  Let
  $v$ be the pivot of $H$, $u$ and $u'$ its antennas, and $w$ its
  bottom.  So, there is an edge-partition of $H$ into a directed path  $P_u$ from $u$
  to $w$, a directed path $P_{u'}$ from $u'$ to $w$ and the edges $uv$
  and $u'v$.

  We claim that $c$ has at most one neighbor in $P_u$.  Otherwise, $c$
  and a subpath $P_u$ form a hole $J$, and since $P_u$ is directed,
  this hole cannot contain two sources, a contradiction. Similarly,
  $P_{u'}$ contains at most one neighbor of $c$.  Hence, the only
  possibility for $c$ to have at least three neighbors in $H$ is that
  $c$ is adjacent to $v$, to one internal vertex of $P_u$ and to one
  internal vertex of $P_{u'}$. Notice that $ c $ cannot be adjacent to $ u $ or $ u' $ otherwise there will be a triangle in $ G $.

  Two holes $H_u$ and $H_{u'}$ of $G$, containing respectively $u$ and
  $u'$, go through the edge $vc$, forming a domino.  Since $c$ is not
  adjacent to the sources of $u$ and $u'$, it can be the pivot of
  neither $H_u$ nor $H_{u'}$.  Hence, by Lemma~\ref{l:domino} $v$ must
  be the pivot of either $H_u$ or $H_{u'}$, say of $H_u$ up to
  symmetry.  Let $x$ be the neighbor of $c$ in $P_u$.  Since $v$ is the
  pivot of $H_u$, $cx$ is an arc of $G$.  Since $x$ is not the pivot
  of $H_u$, by Lemma~\ref{l:domino} $x$ is the pivot of $H_w$,
  that is the hole of $G$ containing $c$ and $w$.  Hence, $x$ is a
  sink of $H_v$, a contradiction to $P_u$ being directed from $u$ to
  $w$.    
\end{proof}

\begin{figure}
	\begin{center}
		\includegraphics[width=10cm]{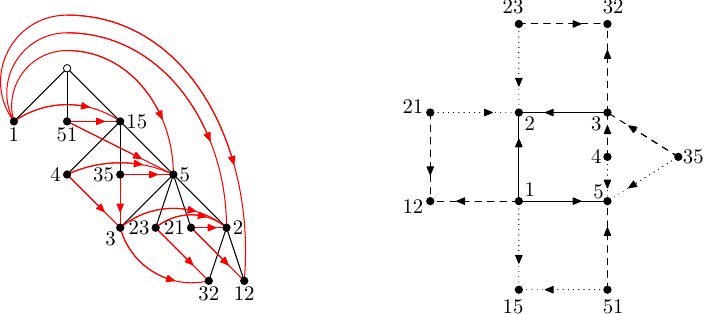}\\\rule{0em}{2ex}\\
		\includegraphics[width=10cm]{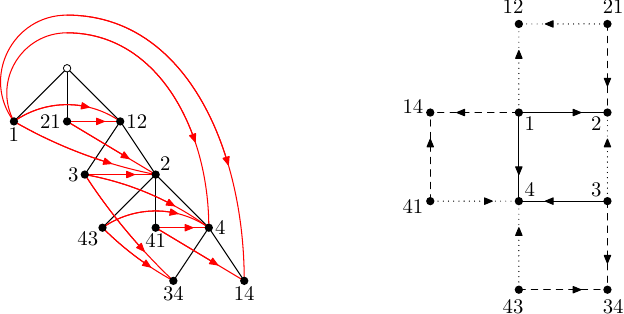}\\\rule{0em}{2ex}\\
		\includegraphics[width=10cm]{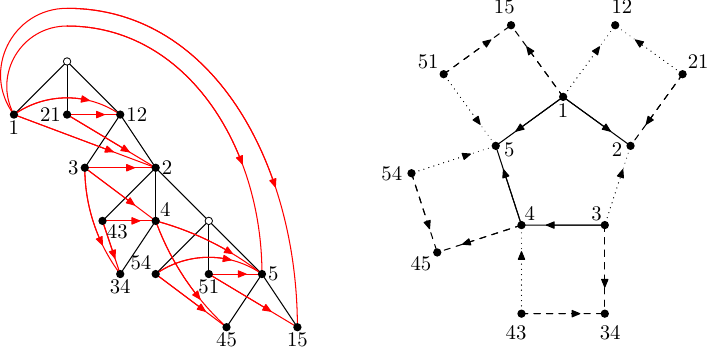}
		
		\caption{Burling graphs close to flowers.\label{f:flowers-6}}
	\end{center}
\end{figure}

\begin{figure}
	\begin{center}
		\includegraphics[width=12cm]{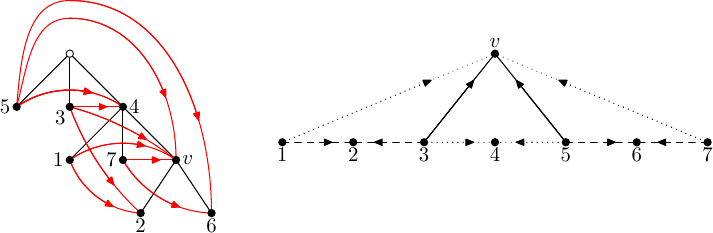}\\\rule{0em}{4ex}\\
		\includegraphics[width=12cm]{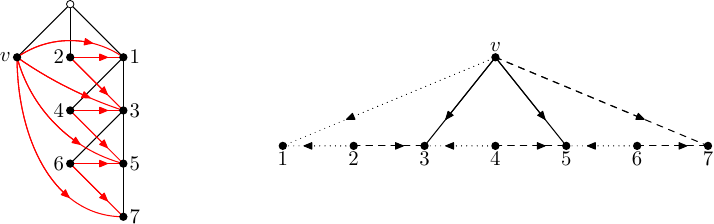}\\\rule{0em}{4ex}\\
		\includegraphics[width=12cm]{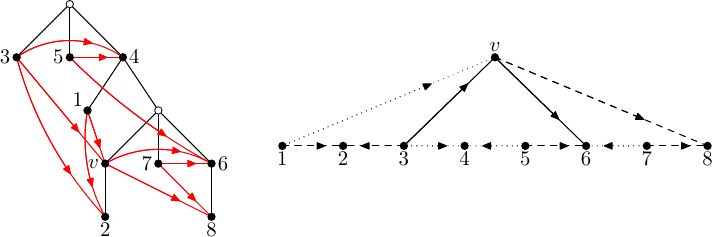}
		
		\caption{Burling graphs close to wheels.\label{f:fan}}
	\end{center}
\end{figure}

As shown in Figures~\ref{f:flowers-6} and~\ref{f:fan}, graphs that are
quite close to flowers or wheels can be Burling graphs.

In \cite{Trotignon13}, Trotignon asked whether the class of wheel-free graphs is $\chi$-bounded (see Question 5.1 of \cite{Trotignon13}). In \cite{ScottSeymour17}, Scott and Seymour made a closely related conjecture that the class of all graphs that for all $ k $ that do not contain an induced cycle such that
some vertex has at least $k$ neighbours on the cycle is $\chi$-bounded (see 12.16 in \cite{ScottSeymour17}). Theorem~\ref{thm:Wheel-not-Burling} answers in negative to the former and disproves the latter.

The next theorem fully characterizes subdivisions of $K_4$ that are Burling
graphs.

\begin{theorem}
  \label{l:K4}
  Let $G$ be a non-oriented graph obtained from $K_4$ by subdividing
  edges.  Then $G$ is a Burling graph if and only if $G$ contains four
  vertices $a$, $b$, $c$ and $d$ of degree~3 such that
  $ab, ac \in E(G)$ and $ad, bc\notin E(G)$.
\end{theorem}

\begin{proof}
  Suppose that $G$ is a Burling graph.  Let $a$, $b$, $c$ and $d$ be
  the four vertices of degree~3 of $G$.  If $G[\{a, b, c, d\}]$
  contains no vertex of degree at least~2, then $G$ is isomorphic to
  one of the graphs represented in Figure~\ref{f:nonBurlingK4-4}, so
  $G$ has no star cutset, a contradiction to Theorem~\ref{th:chalop}.
  So, up to symmetry, we may assume that $a$ has degree at least~2 in
  $G[\{a, b, c, d\}]$, so up to symmetry $ab, ac \in E(G)$.  If
  $bc\in E(G)$, then $G$ contains a triangle, a contradiction to
  Lemma~\ref{lem:no-triangle}.  So, $bc\notin E(G)$.  If $ad\in E(G)$,
  then $G$ is a wheel, a contradiction to
  Lemma~\ref{thm:Wheel-not-Burling}.  So, $ad\notin E(G)$.  We proved
  that $ab, ac \in E(G)$ and $ad, bc\notin E(G)$.

  Conversely, if we suppose that $ab, ac \in E(G)$ and
  $ad, bc\notin E(G)$, then $G$ is obtained by subdividing dashed
  edges of the the graph represented in
  Figure~\ref{f:firstExample-14}.  It is therefore a Burling graph as
  explained after the proof of Lemma~\ref{l:subdivAll}.
\end{proof}


\end{document}